\documentclass[11pt]{amsart}
\usepackage{amsmath}
\usepackage{amsfonts}
\usepackage{amsthm}
\usepackage[english]{babel}
\usepackage{cite}
\usepackage{xcolor}
\usepackage{hyperref}

\theoremstyle{plain}
\newtheorem{thm}{Theorem}
\newtheorem{lem}{Lemma}
\newtheorem{prop}{Proposition}

\theoremstyle{definition}

\newtheorem{eg}{Example}

\theoremstyle{definition}
\newtheorem*{defn*}{Definition}
\newtheorem*{eg*}{Example}
\newtheorem*{rem*}{Remark}

\DeclareMathOperator{\Tr}{Tr}

\begin{document}
	\title[Positivity properties of the vbMA equation]{Positivity properties of the vector bundle Monge--Amp\`ere equation}	
	\author{Aashirwad N. Ballal}
	\address{Department of Mathematics, Indian Institute of Science, Bangalore, India -- 560012}
	\email{aashirwadb@iisc.ac.in}
	\author{Vamsi P. Pingali}
	\address{Department of Mathematics, Indian Institute of Science, Bangalore, India -- 560012}
	\email{vamsipingali@iisc.ac.in}
	\maketitle
	\begin{abstract}
		We study MA-positivity, a notion of positivity relevant to a vector bundle version of the complex Monge--Amp\`ere equation introduced in an earlier work, and show that for rank-two holomorphic bundles over complex surfaces, MA-semi-positive solutions of the vector bundle Monge--Amp\`ere (vbMA) equation are also MA-positive. For vector bundles of rank-three and higher, over complex manifolds of dimension greater than one, we show that this positivity-preservation property need not hold for an algebraic solution of the vbMA equation treated as a purely algebraic equation at a given point. Finally, we set up a continuity path for certain classes of highly symmetric rank-two vector bundles over complex three-folds and prove a restricted version of positivity preservation which is nevertheless sufficient to prove openness along this continuity path.
	\end{abstract}
	
	\section{\textbf{Introduction}}\label{sec:Intro}
\indent 	A vector bundle version of the complex Monge--Amp\`ere (vbMA) equation and a new notion of positivity was introduced in \cite{Pin18}, as a possible way obtaining inequalities involving higher Chern classes on complex manifolds where this equation is satisfied. Let $ V $ be a holomorphic vector bundle over a compact complex $n$-manifold $ M $. If $ h $ is a smooth Hermitian metric on $ V $, there is the associated Chern connection $ \nabla_h $ (or $ D_h $) on $ V $ with curvature $ \Theta_h $. If a volume form $ \eta $ on $ M $ is fixed, the vector bundle Monge--Amp\`ere equation for the metric $ h $ on $ V $ can be written as
	\begin{equation}
		\label{vbMA}
		(i\Theta_h)^n = \eta \otimes \text{Id}.
	\end{equation}	
	
	Note that since $ (2 \pi)^n n!\text{ch}_n(V).M = \int_M \Tr(i \Theta_h)^n $ is independent of the metric $ h $ and depends only on the Chern character class $\text{ch}_n(V)$, we must have $ \text{rank}(V) \int_M \eta = (2 \pi)^n n!\text{ch}_n(V) $ in order for a solution to exist. For line bundles, this equation boils down to the usual complex Monge-Amp\`ere equation, for which existence was proven by Yau \cite{Yau78} in a K\"ahler class. For $n=1$, the equation is simply the Hermitian-Einstein equation on Riemann surfaces, for which a Kobayashi-Hitchin correspondence was proven by Atiyah and Bott \cite{AB83} using the Narasimhan and Seshadri theorem \cite{NS65}, and independently by Donaldson \cite{Don83}. Thus, it is expected that for Equation \ref{vbMA} to be solvable, a stability condition (MA-stability) must hold, and there must exist a background metric whose curvature satisfies a K\"ahler-like positivity condition (MA-positivity -- see Section \ref{sec:MApositivity} for the precise definition). Indeed, the latter requirement in the usual Monge-Amp\`ere equation is paramount to prove estimates as well as openness in any continuity path. In \cite{Pin18}, for certain rank-2 bundles on surfaces, MA-stability was proven to hold as a necessary condition, a Kobayashi-Hitchin correspondence was proven in the special case of vortex bundles of rank-2 on certain surfaces, and a perturbative existence result was proven for Mumford-stable bundles. Since then, the vbMA equation (and its variants) have been an active subject of research \cite{Dem21, DMS20, Gho24, Tak24, ZZ22, Cor23, KS24}.

\indent One can attempt to solve Equation \eqref{vbMA} using a method of continuity as follows: Assume that $V$ is Mumford-stable with respect to an ample line bundle $L$, and let $h_0$ be a Hermitian-Einstein metric with respect to a K\"ahler form $\Omega \in c_1(L)$ obtained by solving $\Omega^n=\eta$ using \cite{Yau78}. Now consider the following continuity path
\begin{gather}
(i\Theta_{h_t}+t\Omega \otimes Id)^n = c_t \Omega^n \otimes\text{Id},
\label{contpath}
\end{gather}
where $c_t= \int_M \text{Tr} (i \Theta_0 + t \Omega \otimes \text{Id})^{n} / (\text{rank}(V) \int_M \Omega^{n}) $ is a polynomial of degree $ n$ in $ t $. That is, we consider the $\mathbb{R}$-vector bundle $V\otimes L^{t/2 \pi}$ and solve the vbMA equation for it with a normalised right-hand-side. It was proven in \cite{Pin18} (and the proof is sketched in Section \ref{sec:threefolds}) that for large $t>>1$, there exists a smooth solution to \eqref{contpath} (which is trivially MA-positive). If we prove that the set of $t\in [0,\infty)$ for which there is a smooth solution is open and closed, we will be done. To this end, it is paramount that MA-positivity be preserved along this continuity path. At the first time this condition fails, the solution will be MA-semipositive. This phenomenon will occur for any continuity path. Thus we are naturally led to the following conjecture.\\

\textbf{Conjecture}: \emph{Any MA-semipositively curved solution of the vbMA equation \ref{vbMA} is MA-positively curved.}\\

\indent In this paper (Theorem \ref{mainthm}) we prove that for rank-2 bundles on surfaces this conjecture holds. Surprisingly enough, we exhibit a counterexample to the pointwise, purely linear-algebraic version of this conjecture for rank-3 (and higher rank) bundles on complex surfaces (and in higher dimensions). This counterexample uses the construction of a vortex bundle \cite{Gar93,Gar94}. We then proceed to study a rank-2 vortex bundle on certain threefolds to formulate a dimensional reduction of the vbMA equation (Equation \ref{vbMA3d}), and set up a continuity method for it. To prove openness, we need a restricted version of the conjecture mentioned above, and we are able to prove the same. The proofs are complicated and use the Schur complement of a block matrix extensively.

\indent Our counterexample indicates that either the vbMA equation is the ``wrong" equation to study for higher ranks, or that simple linear algebra is not sufficient and delicate analytic estimates (see \cite{Man23} for a related phenomenon for the Demailly system) are necessary to prove that MA-positivity is preserved. This observation has wider ramifications. Indeed, motivated by mirror symmetry considerations, the deformed Hermitian-Yang-Mills equation was introduced by \cite{MMMS00,LYZ00,CJY20} for line bundles. It is conjecturally mirror to a special Lagrangian section (see \cite{CXY18} for an overview). One can wonder whether a vector bundle version of it can be mirror to a special Lagrangian multi-section. A naive generalisation of the same \cite{CY18} is equivalent to the vbMA equation in the case of surfaces by completing the square \cite{JY17, KS24}. Our results show that perhaps the naive generalisation is extremely subtle at best, or too naive at worst. On the positive side, the preservation of MA-positivity for rank-2 bundles on surfaces indicates that a metric version of a theorem of Schneider and Tancredi \cite{ST85} can be approached using the continuity path in Equation \eqref{contpath}. If such a result is proven, it will provide strong evidence for the Griffiths conjecture relating ampleness of vector bundles and Griffiths-positively curved metrics \cite{Gri69}.

\emph{Acknowledgements}: This work is partially supported by grant F.510/25/CAS-II/2018(SAP-I) from UGC (Govt. of India) and a scholarship from the Indian Institute of Science. The authors thank Ved Datar for his support and encouragement.

	\section{\textbf{The vbMA equation and MA-positivity}}\label{sec:MApositivity}
	
 In this section, we discuss the notion of MA-positivity and also state our main theorem. As mentioned in Section \ref{sec:Intro}, to produce solutions to equations such as Equation \eqref{vbMA}, one might consider using a continuity method along a path of metrics $ t \mapsto h_t $ satisfying equations similar to the vbMA equation. In such cases, questions about the mapping properties of the linearisation the vbMA equation naturally arise whilst studying the openness aspect of the continuity method. With this observation in mind, we recall the following MA-positivity condition defined in \cite{Pin18}:
	
	\begin{defn*}[MA-positivity]
		Given a Hermitian metric $ h $ on $ V $ and a Hermitian section $ M $ of $ \Omega^{1, 1}(\text{End}(V)) $ is MA-positive at a point $ p $ if
		$$\sum_{k = 0}^{n-1} \Tr(ia \wedge (M)^k \wedge a^{*_h} \wedge (M)^{n - 1 - k} ) > 0 $$at $ p $ for all $ 0 \neq a \in \Omega^{1, 0} (\text{End(V)}) $. If this condition holds at all points, we say that $ M $ is MA-positive.
	\end{defn*}

	\begin{rem*}If $ h_0 $ is some Hermitian metric on $ V $, then any other Hermitian metric is of the form $ h = h_0 e^g  $, where $ g \in \text{End}(V, h_0) $, i.e. g is a $ h_0 $-Hermitian section of $ \text{End}(V) $. Also, if $ \Theta_h $ is the curvature of the Chern connection $ \nabla_h $, then it can be shown that $ \Theta_h = \Theta_{h_0} + \nabla^{0, 1}(e^{-g} \nabla^{1, 0}_{h_0} e^{g}) $, where we have omitted the subscript in $ \nabla^{0, 1} $ since $ \nabla^{0, 1}_h = \nabla^{0, 1}_{h_0} $.
	\end{rem*}

	To see how MA-positivity can be applied in the context of the vbMA equation, consider the linearization of $ (i \Theta_{h_0})^n $ at $ h_0 $. By the above remark, this linearization is
	\begin{align*}
		L_{h_0}: \text{End}(V, h_0) \rightarrow \Omega^{n, n} (\text{End(V)}) \\
		g \mapsto \sum_{k = 0}^{n-1} (i\Theta_{h_0})^k \wedge i \nabla^{0, 1} \nabla_{h_0}^{1, 0} g  \wedge (i \Theta_{h_0})^{n - 1 - k}.
	\end{align*}
	Using the pairing $\Omega^{n, n} (\text{End(V)}) \times \text{End(V)} \rightarrow \Omega^{n , n}(M) $ given by $ (X, Y) \mapsto \Tr(XY^{*_{h_0}}) $ and choosing a volume form, we can identify $ \Omega^{n, n} (\text{End(V)}) $ with $ \text{End(V)} $ and also define an inner product on $ \text{End(V)} $. Doing this, we get $ L_{h_0}: \text{End}(V, h_0) \rightarrow \text{End}(V) $.  To usefully apply Fredholm operator theory to $ L_{h_0} $, some information about $ \text{ker}(L_{h_0}) $ and $ \text{ker}(L^*_{h_0}) $ is necessary. In general, the domain of $ L^* $ is not the same as the domain of $ L $, but when $ \Theta_{h_0} $ satisfies the vbMA equation, we have the following proposition.
	
	\begin{prop}
		If $ \Theta_{h_0} $ satisfies the vbMA equation, then $ \forall g \in \text{End}(V, h_0) $, $ L_{h_0} g $ is also in $ \text{End}(V, h_0) $, so $ L_{h_0} $ can be considered as an operator on $ \text{End}(V, h_0) $, which is also formally self-adjoint.
	\end{prop}
\begin{proof}
	We have $ (L_{h_0}(g))^{*_{h_0}} $ = -$\sum_{k = 0}^{n-1} (i\Theta_{h_0})^k \wedge i \nabla^{1, 0}_{h_0} \nabla^{0, 1} g  \wedge (i \Theta_{h_0})^{n - 1 - k}$ and so
	\begin{multline*}
	L_{h_0}(g) - (L_{h_0}(g))^{*_{h_0}} = i \sum_{k = 0}^{n-1} (i\Theta_{h_0})^k \wedge  (\nabla^{0, 1} \nabla^{1, 0}_{h_0} + \nabla^{1, 0}_{h_0} \nabla^{0, 1}) g  \wedge (i \Theta_{h_0})^{n - 1 - k}  \\
	= i \sum_{k = 0}^{n-1} (i\Theta_{h_0})^k \wedge  (\nabla_{h_0} \wedge \nabla_{h_0}) g  \wedge (i \Theta_{h_0})^{n - 1 - k} \\
	= i \sum_{k = 0}^{n-1} (i\Theta_{h_0})^k \wedge [\Theta_{h_0}, g]  \wedge (i \Theta_{h_0})^{n - 1 - k}
	= [(i\Theta_{h_0})^n, g] = \eta \cdot [\text{Id}, g] = 0.
	\end{multline*}The formal self-adjointness follows from integrating-by-parts and using the Bianchi identity $ \nabla_{h_0} \Theta_{h_0} = 0 $.
\end{proof}
	The ellipticity of $ L $ follows from MA-positivity by taking the endomorphism-valued (1, 0)-forms in the definition of MA-positivity to be of the form $ \xi \wedge g $, where $ \xi $ is a $ (1, 0) $-form and $ g $ is an endomorphism. Hence, by the theory of self-adjoint elliptic operators, $ L_{h_0} $ is an isomorphism on the orthogonal complement of $ \text{ker}(L_{h_0}) $. Further, as in \cite{Pin18}[Lemma 2.3], by an integration-by-parts, MA-positivity of $ i\Theta_{h_0} $ implies that $ g \in \text{ker}(L_{h_0}) \iff \nabla_{h_0} g = 0 $. Indeed, for any $ g \in \text{End}(V, h_0)  $, $ \displaystyle \int_M \Tr (L_{h_0}(g) g) = - \int_M \sum_{k = 0}^{n-1} \Tr(i \nabla_{h_0}^{0, 1} g \wedge (i \Theta_{h_0})^k \wedge \nabla_{h_0}^{1, 0} g \wedge (i \Theta_{h_0})^{n - 1 - k} ) $, which is $ 0 $ iff $ \nabla_{h_0} g  = 0$. In the same Lemma, it is also noted that when $ V $ is indecomposable as a holomorphic bundle, this kernel consists of constant multiples of the identity so that $ L_{h_0} $ is an isomorphism on the subspace of $ \text{End}(V, h_0) $ consisting of endomorphisms whose integrated trace is zero. This is usually sufficient to prove openness along the relevant continuity paths.
	\begin{rem*}Note that although we have required positivity for all non-zero $ (1, 0) $-valued endomorphisms in the definition of MA-positivity, what we have really used in the above discussion is MA-positivity for  $ (1, 0) $-valued endomorphisms of the type $ \nabla_{h_0} g $. In general, this is no restriction at all on $ \Omega^{1, 0}(\text{End(V)}) $, but when a vector bundle has some symmetries, such as the vortex bundles we consider later, we can restrict the $ (1, 0) $-valued endomorphisms used in the definition of MA-positivity to a proper subspace of $ \Omega^{1, 0}(\text{End(V)}) $ containing all endomorphisms of the type $ \nabla_{h_0} g $ for the relevant $ g $. In those cases, we refer to the corresponding (semi)-positivity condition as restricted-MA-(semi)-positivity.
	\end{rem*}
	It is also important for the method of continuity that MA-positivity be preserved in taking limits, in the sense that if $ \Theta $ is a solution of a vbMA-like equation obtained by taking a limit of MA-positive solutions of vbMA-like equations, then $ \Theta $ must also be MA-positive. If the limit is taken in $ C^{2} $, then continuity necessitates that $ \Theta $ is MA-semi-positive, in the sense that $ \Theta $ satisfies MA-positivity with the strict inequality replaced by a non-strict inequality. So it is important in proving openness of the continuity path to know if a solution of the vbMA equation which is MA-semi-positive is also MA-positive. We now consider the local version of this problem, which, if solved, would also solve the global version. \\
		
	Let $ \Lambda^{p, q}(\mathbb{C}^n, \text{End}(\mathbb{C}^r)) $ be the set of $ \text{End}(\mathbb{C}^r)  $-valued $ (p, q) $-forms on $ \mathbb{C}^n $ and let $ \eta $ be a volume form on $ \mathbb{C}^n  $.
	\begin{defn*}[algebraic vbMA solutions, MA-positivity, MA-semi-positivity] $$\text{(1) The set of all solutions to the algebraic vbMA equations: }  $$ $$ V(n, r, \eta ) := \{ \Theta \in \Lambda^{1, 1}(\mathbb{C}^n, \text{End}(\mathbb{C}^r)): i\Theta = (i\Theta)^{*}, (i\Theta)^n = \eta \cdot \text{Id}  \} $$  
	$$\text{(2) The set of MA-positive forms: } $$ $$ P(n, r) := \{ \Theta \in \Lambda^{1, 1}(\mathbb{C}^n, \text{End}(\mathbb{C}^r)):  \forall a \in \Lambda^{1, 0}(\mathbb{C}^n, \text{End}(\mathbb{C}^r)), a \neq 0,  $$ $$ \sum_{k = 0}^{n-1} \Tr(ia \wedge (i \Theta)^k \wedge a^{*} \wedge (i\Theta)^{n - 1 - k} ) > 0 \}  $$
	$$\text{(3) The set of MA-semi-positive forms: }  $$ $$ S(n, r) := \{ \Theta \in \Lambda^{1, 1}(\mathbb{C}^n, \text{End}(\mathbb{C}^r)): \forall a \in \Lambda^{1, 0}(\mathbb{C}^n, \text{End}(\mathbb{C}^r)),  $$ $$ \sum_{k = 0}^{n-1} \Tr(ia \wedge (i \Theta)^k \wedge a^{*} \wedge (i\Theta)^{n - 1 -  k} ) \geq 0 \}   $$
\end{defn*}
\begin{rem*}
	As mentioned in the previous remark, in some cases, it is also useful to consider MA-positivity on proper subspaces $ W $ of $ \Lambda^{1, 0}(\mathbb{C}^n, \text{End}(\mathbb{C}^r)) $ and in those cases, we deal with proper subsets $ P_W(n, r) $ and $ S_W(n, r) $ of $ P(n, r) $ and $ S(n, r) $.
\end{rem*}

The question of MA-positivity preservation is whether $ V(n, r, \eta ) \cap S(n, r) \subseteq P(n, r) $ for all $ n, r, \eta $. If either $ n = 1 $ or $ r = 1 $ (the case relevant for the ordinary complex Monge--Amp\`ere equation), this inclusion holds. As mentioned in the introduction, we prove in this paper that it also holds when $ n = r = 2 $ and that it fails to hold if $ n \geq 2 $ and $ r > 2 $, so unqualified MA-positivity preservation does not, in fact, hold. However, it could still be the case that for the proper subset of $ V(n, r, \eta) $ consisting of those endomorphism-valued $ (1, 1) $-forms which arise as genuine solutions of vbMA equations, MA-semi-positivity does indeed imply MA-positivity. The final result of the paper shows that for any solution of the vbMA equation on certain rank-$2$ vortex bundles over complex three-folds, restricted-MA-semipositivity does imply restricted-MA-positivity. In summary, we have the following theorem:

\begin{thm}
	\label{mainthm}
	~\\
	(1) MA-positivity is preserved on rank-2 bundles over complex surfaces. In fact, $V(2, 2, \eta ) \cap S(2, 2) \subseteq P(2, 2)$. \\
	(2) For $ n \geq 2 $ and $ r > 2 $, it is not true that $V(n, r, \eta ) \cap S(n, r) \subseteq P(n, r)$ in general. \\
	(3) Restricted-MA-positivity is preserved on a class of rank-2 vortex bundles over complex three-folds.
\end{thm}

	\section{\textbf{MA-positivity for rank-2 bundles on surfaces}}\label{sec:rank2surfaces}
	Our aim in this section is to show that an MA-positive-semidefinite solution of the vbMA equation on a rank-2 bundle over a complex surface is, in fact, MA-positive-definite. As shown in the previous section, it suffices to consider the local version of this problem. \\
	\indent We may write the curvature as $i\Theta = A idz^1\wedge d\bar{z}^1 + C idz^2 \wedge d\bar{z}^2 + B idz^1\wedge d\bar{z}^2 + B^{\dag} idz^2 \wedge d\bar{z}^1$ where $A, B, C$ are $2\times 2$ complex matrices (with $A=A^{\dag}$, $C=C^{\dag}$). Now suppose $a^{\dag}$ is an $2\times 2$ matrix of $(1,0)$ forms given by $a^{\dag}=\alpha dz^1 + \beta dz^2$ where $\alpha, \beta$ are $2\times 2$ matrices of complex numbers. Then we see that
	\begin{gather}
	\frac{\mathrm{tr} \left (ia^{\dag} \left ( i\Theta \right ) a \right )+\mathrm{tr} \left (ia^{\dag}  a  \left ( i\Theta \right )\right )}{idz^1d\bar{z}^1 idz^2 d\bar{z}^2} = \mathrm{tr}(\alpha C \alpha^{\dag})+\mathrm{tr}(\alpha \alpha^{\dag}C)+\mathrm{tr}(\beta A \beta^{\dag})+\mathrm{tr}(\beta  \beta^{\dag} A) \nonumber \\
	-\mathrm{tr}(\alpha B^{\dag} \beta^{\dag})-\mathrm{tr}(\alpha  \beta^{\dag} B^{\dag}) -\mathrm{tr}(\beta B \alpha^{\dag}) - \mathrm{tr}(\beta  \alpha^{\dag} B).
	\label{MApositivityequationintermsofmatrices}
	\end{gather}
	On the other hand,
	\begin{gather}
	(i\Theta)^2 = idz^1 d\bar{z}^1 dz^2 d\bar{z}^2 (AC+CA-BB^{\dag}-B^{\dag}B)>0.
	\label{iThetasquaredintermsofmatrices}
	\end{gather}
	Note that $\{B,B^{\dag}\}=BB^{\dag}+B^{\dag}B$ is positive-semidefinite. Thus, $\{A,C\}$ is positive-definite. Equation \eqref{MApositivityequationintermsofmatrices} shows that MA-positive-semidefiniteness implies the positive-semidefiniteness of $A,C$. Since $\{A,C\}$ is positive-definite we see that $A,C$ are positive-definite. Note that the right-hand-side of Equation \eqref{MApositivityequationintermsofmatrices} can be written as follows.
	\begin{gather}
	\frac{\mathrm{tr} \left (ia^{\dag} \left ( i\Theta \right ) a \right )+\mathrm{tr} \left (ia^{\dag}  a  \left ( i\Theta \right )\right )}{idz^1d\bar{z}^1 idz^2 d\bar{z}^2} = \mathrm{tr} \left( \left[ \begin{array}{cc} \alpha^{\dag} & \beta^{\dag} \end{array} \right] \left[ \begin{array}{cc} \{C,\} & -\{B, \} \\  -\{B^{\dag}, \}  & \{A, \} \end{array} \right ] \left[ \begin{array}{c} \alpha \\ \beta \end{array} \right] \right),
	\label{writingusingblockmatrices}
	\end{gather}
	where $\{X, \}$ is the linear map $Y\rightarrow XY+YX$. At this juncture, we have the following lemma.
	\begin{lem}\label{lem:positivityofanticommutator}
		The map $Y\rightarrow CY+YC$ is positive-definite.
	\end{lem}
	\begin{proof}
		\begin{gather}
		\mathrm{tr}\left(Y^{\dag}(CY+YC) \right)= \frac{\mathrm{tr} \left (ia^{\dag} \left ( i\Theta \right ) a \right )+\mathrm{tr} \left (ia^{\dag}  a  \left ( i\Theta \right )\right )}{idz^1d\bar{z}^1 idz^2 d\bar{z}^2} \geq 0,
		\end{gather}
		where $\alpha=Y$ and $\beta=0$. Now $Y\rightarrow \{C,Y\}$ has trivial kernel (a fact that can be easily proven by diagonalising $C$). Hence this map is positive-definite.
	\end{proof}
	Let $\{C,\}^{-1}(Y)=T(Y)$. By Schur's theorem, MA-positive-definiteness holds iff the Schur complement of the matrix in Equation \eqref{writingusingblockmatrices} is positive-definite, i.e.,
	\begin{gather}
	\mathrm{tr}\left(Y^{\dag} (\{A,Y\} -\{B^{\dag},T(\{B,Y\})\})\right) > 0 \ \forall \ Y \neq 0.
	\label{positivityintermsofSchur}
	\end{gather}
	Without loss of generality, we can take $ C $ to be $$ \begin{bmatrix}
	1 & 0 \\ 0 & \lambda
	\end{bmatrix}$$ for some $ \lambda > 0 $. It then follows that $ \displaystyle T(X) = \begin{bmatrix}
	\frac{X_{11}}{2} & \frac{X_{12}}{1 + \lambda} \\ \frac{X_{21}}{1 + \lambda} & \frac{X_{22}}{2 \lambda}
	\end{bmatrix} $ for $ \displaystyle X = \begin{bmatrix}
	X_{11} & X_{12} \\ X_{21} & X_{22}
	\end{bmatrix}  $. \\
	
	Denote the inner product $ (X, Y) \mapsto \mathrm{tr}(X^{\dag} Y) $ by $ \langle X, Y \rangle $ and the inner product  $ (X, Y) \mapsto \mathrm{tr}(X^{\dag}T(Y)) $ by $ \langle A, B \rangle_T $. Also, denote $ \{ X, Y \} = XY + YX $ by $ X \cdot Y $. \\
	
	Since $ A > T(\{B, B^\dag \}) $, by using the fact that $ \mathrm{tr} $ is cyclic, it suffices to show
	\begin{equation}
	\label{MASchur}
	\langle B \cdot B^{\dag}, X \cdot X^{\dag} \rangle_T \geq \langle B \cdot X^{\dag}, B \cdot X^{\dag} \rangle_T
	\end{equation}for all $ B, X \in M_2(\mathbb{C}) $ and any positive $ \lambda $. \\
	
	For any matrices $ P, Q $,
	\begin{equation*}
	\langle P, Q \rangle_T = \sum_{i, j} \frac{\overline{P_{ij}} Q_{ij}}{\lambda_i + \lambda_j}
	\end{equation*}where $ \lambda_1 = 1 $, $ \lambda_2 = \lambda $. So the inequality $ \ref{MASchur} $ can be written (after multiplying both sides by $ 2 \lambda (1 + \lambda) $) as
	\begin{equation}
	\label{quadraticineq}
	\lambda^{2} a_{11} + \lambda (a_{11} + a_{22} + 2(a_{12} + a_{21})) + a_{22} \geq 0
	\end{equation}where
	\begin{equation*}
	a_{ij} = (B \cdot B^{\dag})_{ij} (X \cdot X^{\dag})_{ji} - |(B \cdot X^{\dag})|_{ij}^{2}.
	\end{equation*}
	
	We now compute the matrix elements of $ B \cdot X^{\dag} $
	\begin{multline}
	\label{matrixcomputation}
	BX^{\dag} + X^{\dag}B = \begin{bmatrix}
	B_{11} & B_{12} \\ B_{21} & B_{22}
	\end{bmatrix}\begin{bmatrix}
	\overline{X}_{11} & \overline{X}_{21} \\ \overline{X}_{12} & \overline{X}_{22}
	\end{bmatrix} + \begin{bmatrix}
	\overline{X}_{11} & \overline{X}_{21} \\ \overline{X}_{12} & \overline{X}_{22}
	\end{bmatrix} \begin{bmatrix}
	B_{11} & B_{12} \\ B_{21} & B_{22}
	\end{bmatrix} \\
	= \begin{bmatrix}
	2B_{11}\overline{X}_{11} + B_{12}\overline{X}_{12} + B_{21} \overline{X}_{21} & \overline{X}_{21}(B_{11} + B_{22}) + B_{12}(\overline{X}_{11} + \overline{X}_{22}) \\ 2B_{22}\overline{X}_{22} + B_{12}\overline{X}_{12} + B_{21} \overline{X}_{21} & \overline{X}_{12}(B_{11} + B_{22}) + B_{21}(\overline{X}_{11} + \overline{X}_{22})
	\end{bmatrix}.
	\end{multline}
	Using this, we have
	\begin{multline}
	a_{11} = (2|B_{11}|^{2} + |B_{12}|^{2} + |B_{21}|^{2})(2|X_{11}|^{2} + |X_{12}|^{2} + |X_{21}|^{2}) \\
	- | 2B_{11} \overline{X}_{11} + B_{12} \overline{X}_{12} + B_{21} \overline{X}_{21}|^{2}
	\end{multline}Due to the identity
	\begin{equation*}
	|v|^{2}|w|^{2} - |\langle v, w \rangle |^{2} = \sum_{i < j}|v_iw_j - w_j v_i|^{2},
	\end{equation*}we get
	\begin{equation}
	a_{11} = 2|B_{11}X_{12} - B_{12} X_{11}|^{2} + 2|B_{11}X_{21} - B_{21} X_{11}|^{2} + |B_{12}X_{21} - B_{21} X_{12}|^{2}.
	\end{equation}
	Similarly,
	\begin{equation}
	a_{22} = 2|B_{22}X_{12} - B_{12} X_{22}|^{2} + 2|B_{22}X_{21} - B_{21} X_{22}|^{2} + |B_{12}X_{21} - B_{21} X_{12}|^{2}.
	\end{equation}
	
	Next, we compute $ a_{12} + a_{21} $ using $ \ref{matrixcomputation} $ again:
	\begin{equation*}
	|B \cdot X^{\dag}|_{12}^{2} = |X_{21}|^{2}|B_{11} + B_{22}|^{2} + |B_{12}|^{2}|X_{11} + X_{22}|^{2} + 2 \Re((B_{11} + B_{22})(X_{11} + X_{22})\overline{X}_{12} \overline{B}_{12})
	\end{equation*}and
	\begin{equation*}
	|B \cdot X^{\dag}|_{21}^{2} = |X_{12}|^{2}|B_{11} + B_{22}|^{2} + |B_{21}|^{2}|X_{11} + X_{22}|^{2} + 2 \Re((B_{11} + B_{22})(X_{11} + X_{22})\overline{X}_{21} \overline{B}_{21}).
	\end{equation*}
	We also have
	\begin{align*}
	(B \cdot B^{\dag})_{12}(X \cdot X^{\dag})_{21} = &
	(B_{11} + B_{22}\overline{B}_{21} + (\overline{B}_{11} + \overline{B}_{22})B_{12}) + (X_{11} + X_{22}\overline{X}_{12} + (\overline{X}_{11} + \overline{X}_{22})X_{21}). \\
	\end{align*}
	So
	\begin{align*}
	(B \cdot B^{\dag})_{12}(X \cdot X^{\dag})_{21} + (B \cdot B^{\dag})_{21}(X \cdot X^{\dag})_{12} = & 2 \Re (B \cdot B^{\dag})_{12}(X \cdot X^{\dag})_{21}  \\ = &  2 \Re((B_{11} + B_{22})(X_{11} + X_{22})(\overline{B}_{21}\overline{X}_{12} + \overline{B}_{12}\overline{X}_{21})) \\ + &
	2 \Re((B_{11} + B_{22})(\overline{X}_{11} + \overline{X}_{22})(\overline{B}_{21}{X}_{21} + \overline{B}_{12}{X}_{12}))
	\end{align*}and hence (note that the `real part' term in $ |B \cdot X^{\dag}|_{12}^{2} + |B \cdot X^{\dag}|_{12}^{2}  $ is exactly cancelled by the first `real part' term in the previous equation)
	\begin{equation}
	-(a_{12} + a_{21}) = |(B_{11} + B_{22})X_{21} - (X_{11} + X_{22})B_{21}|^{2} + |(B_{11} + B_{22})X_{12} - (X_{11} + X_{22})B_{12}|^{2}.
	\end{equation} From the above computations, we see that
	\begin{gather*}
	a_{11} = 2|\alpha|^{2} + 2|\beta|^{2} + |\gamma|^{2} \\
	a_{22} = 2|\delta|^{2} + 2|\varepsilon|^{2} + |\gamma|^{2} \\
	-(a_{12} + a_{21}) = |\alpha + \delta|^{2} + |\beta + \varepsilon|^{2}
	\end{gather*}where
	\begin{gather*}
	\alpha = B_{11}X_{12} - B_{12} X_{11} \\
	\beta = B_{11}X_{21} - B_{21} X_{11} \\
	\gamma = B_{12}X_{21} - B_{21} X_{12} \\
	\delta = B_{22}X_{12} - B_{12} X_{22} \\
	\varepsilon = B_{22}X_{21} - B_{21} X_{22}.
	\end{gather*}
	
	As $ a_{11}, a_{22} \geq 0 $, the quadratic $ \ref{quadraticineq} $ is non-negative for all $ \lambda > 0 $ if $ a_{11} + a_{22} + 2(a_{21} + a_{12}) \geq 0 $. On the other hand, if $  a_{11} + a_{22} + 2(a_{21} + a_{12}) < 0 $, then the quadratic is positive for all $ \lambda > 0 $ iff $ |a_{11} + a_{22} + 2(a_{21} + a_{12})| = -(a_{11} + a_{22} + 2(a_{21} + a_{12}) \leq 2 \sqrt{a_{11} a_{22}} $. By the observation above,
	\begin{gather}
	-2(a_{21} + a_{12}) - a_{11} - a_{22} \leq 4(|\alpha \delta| + |\beta \varepsilon|) \\
	\leq 4(\sqrt{(|\alpha|^{2} + |\beta|^{2})(|\delta|^{2} + |\varepsilon|^{2})}) \leq 2\sqrt{a_{11}a_{22}}
	\end{gather}so the quadratic \ref{quadraticineq} is indeed non-negative, completing the proof.
	
	\section{\textbf{The vbMA equation on vortex bundles over complex surfaces}}\label{sec:rankthree}
	In this section, we first show that for $ k \geq 3 $, $\text{End}(\mathbb{C}^k)$-valued $ (1, 1) $-forms satisfying the algebraic vbMA equation on $ \mathbb{C}^2 $ and MA-semi-positivity need not be MA-positive and then extend that result to dimensions greater than 2. Since the (1, 1)-forms we consider will have the algebraic form of curvature endomorphisms of certain vortex bundles over complex surfaces, we start by reviewing the definition of these vortex bundles (one can refer to \cite{Gar93} for more details). Let $ X $ be a Riemann surface and let $ E_1 $ and $ E_2 $ be holomorphic vector bundles over $ X $. Let $ M = X \times  \mathbb{CP}^1  $ and let $ \pi_1 $ and $ \pi_2 $ be the projections onto $ X $ and $  \mathbb{CP}^1  $, respectively.
	
	\begin{defn*}[Vortex bundles]
		As a $ C^\infty $ bundle, $ V := V_1 \oplus V_2 $, where $$ V_1 = \pi_1^{*}E_1 \otimes r \pi_2^{*}\mathcal{O}(2) $$ $$ V_2 = \pi_1^{*}E_2 \otimes (r + 1) \pi_2^{*}\mathcal{O}(2), $$ where $ r $ is a positive integer. The holomorphic structure on $ V $ is not the direct sum of the structures on $ V_1 $ and $ V_2 $ but is rather induced by a section $ T $ of $ \text{Hom}(E_2, E_1) $ as explained below.
	\end{defn*}
	
	$ SU(2) $ acts on a natural way on $ \mathbb{CP}^{1} $ and we extend this action to $ M $ by allowing it to act trivially on $ X $ (and vector bundles over $ X $). One can show, as in \cite{Gar93} and \cite{Gar94}, that all $ SU(2) $-invariant metrics on $ V $ are of the form $ (H_1 \otimes h_{FS}^{r}) \oplus (H_2 \otimes h_{FS}^{(r + 1)}) =: h_1 \oplus h_2 $ where $ H_1 $, $ H_2 $ are Hermitian metrics on $ E_1, E_2 $, and $ h_{FS} $ is a metric on $ \mathcal{O}(2) $ whose curvature is the (1, 1)-form $ \displaystyle - i 2\omega_{FS} :=  2\frac{dz \wedge d\bar{z}}{(1 + |z|^2)^2} $. Since $ \mathcal{O}(2) = T^{1, 0}\mathbb{CP}^{1} $, we can choose $ h_{FS} $ to be the Fubini-Study metric $ \displaystyle h_{FS} = \frac{dz \otimes d\bar{z}}{(1 + |z|^2)^2} $. \\
	
	Let $ T $ be a holomorphic section of $ \text{Hom}(E_2, E_1) $. Consider the connection on $ V $ defined as
	
	$$ D_h =
	\begin{pmatrix}
	D_1 & \beta \\
	-\beta^{*} & D_2
	\end{pmatrix},$$
	
where $ D_i $ is the Chern connection on $ V_i $ for the metric $ h_i $ and $$ \beta = \pi_1^{*} T \otimes \pi_2^{*}\left(\frac{dz}{(1 + |z|^2)^2} \wedge d\bar{z}\right) \in H^{0, 1}(M, \text{Hom}(V_2, V_1) ).$$ Here, $ \beta^{*}$ is the adjoint of $ \beta $ taken using the induced Hermitian form on $ \text{Hom}(V_2, V_1) $. Now $ D_{h}^{0, 1} $ is clearly independent of the choice of metrics $ h $ and it can also be checked that $ D_{h}^{0, 1} \wedge D_{h}^{0, 1} = 0 $, so this connection defines a holomorphic structure on $ V $ for which it is the Chern connection and furthermore, this holomorphic structure is independent of $ h $. This defines the holomorphic structure on $ V $. \\
	
	The curvature of $ D_h $ is
	$$  \Theta_h =
	\begin{pmatrix}
	\Theta_{h_1} - \beta \wedge \beta^* & D^{1, 0} \beta \\
	-D^{0, 1} \beta^{*} & \Theta_{h_2} - \beta^* \wedge \beta
	\end{pmatrix} = \begin{pmatrix}
	\Theta_{1} + (2r + TT^*) (-i\omega_{FS}) & (D^{1, 0} T \otimes \frac{dz}{(1 + |z|^2)^2}) \wedge d\bar{z}) \\
	-(D^{0, 1} T^{*} \otimes \frac{\partial}{\partial z}) \wedge dz  & \Theta_{2} + (2r + 2 - T^*T)(-i\omega_{FS})
	\end{pmatrix} ,$$ where $ \Theta_{1} $ and $ \Theta_{2} $ are the curvatures of the Chern connections on $ E_1 $ and $ E_2 $ defined by $ H_1$ and $ H_2 $, respectively. \\
	
	Hence,
	$$  i\Theta_h \wedge i\Theta_h =
	\begin{pmatrix}
	\{ i\Theta_{1}, (2r + TT^*) \} - iD^{1, 0}T \wedge D^{0, 1} T^* & 0 \\
	0 & \{ i\Theta_{2}, (2r + 2 - T^*T) \} +  iD^{0, 1} T^* \wedge D^{1, 0}T
	\end{pmatrix} \wedge \omega_{FS}, $$ and the vbMA equation on $ V $  becomes
	\begin{align}
	\{ i\Theta_{1}, (2r + TT^*) \} - iD^{ 1, 0}T \wedge D^{0, 1} T^* = \eta \cdot \text{Id}_1 \\
	\{ i\Theta_{2}, (2r + 2 - T^*T) \} +  iD^{0, 1} T^* \wedge D^{1, 0}T = \eta \cdot \text{Id}_2
	\end{align} for some volume form $ \eta $ on $ X $.
	
	Let $ \displaystyle \mathcal{A} = \begin{pmatrix}
	\alpha & \beta \\ \gamma & \delta
	\end{pmatrix} \in \Omega^{1, 0}(\text{End}(V))
	$. By the definition of MA-positivity, $ \Theta_h $ is MA-(semi-)positive iff $ \Tr(i\mathcal{A} \wedge \mathcal{A}^* \wedge i\Theta_h) + \Tr(i\mathcal{A} \wedge i\Theta_h \wedge \mathcal{A}^* ) > 0$ for all $ \mathcal{A} \neq 0 $. We now proceed to compute these wedge products. Firstly,
	$$
	\mathcal{A} \wedge \mathcal{A}^* = \begin{pmatrix}
	\alpha \wedge \alpha^* + \beta \wedge \beta^* & \alpha \wedge \gamma^* + \beta  \wedge \delta^*  \\ \gamma \wedge \alpha^* + \delta \wedge \beta^* & \delta \wedge \delta^* + \gamma \wedge \gamma^*
	\end{pmatrix}.
	$$
	Writing
	$$
	\Theta_h = \begin{pmatrix}
	\Theta_{1} + M  & N \\ -N^* & \Theta_2 + L
	\end{pmatrix}
	,$$ we see that
	\begin{multline}
	\label{AATheta}
	\Tr(\mathcal{A} \wedge \mathcal{A}^* \wedge \Theta_h) = \Tr((\beta \wedge \beta^* + \alpha \wedge \alpha^*) \wedge \Theta_1 ) + \Tr((\beta \wedge \beta^* + \alpha \wedge \alpha^*) \wedge M ) \\
	+ \Tr(\gamma \wedge \alpha^* \wedge N) - \Tr(\alpha \wedge \gamma^* \wedge N^*) \\
	+ \Tr(\delta \wedge \beta^* \wedge N) - \Tr(\beta \wedge \delta^* \wedge N^*) \\
	+ \Tr((\gamma \wedge \gamma^* + \delta \wedge \delta^*) \wedge \Theta_2 ) + \Tr((\gamma \wedge \gamma^* + \delta \wedge \delta^*) \wedge L ).
	\end{multline}
	Also,
	$$
	\mathcal{A} \wedge \Theta_h =
	\begin{pmatrix}
	\alpha \wedge \Theta_1 + \alpha \wedge M - \beta \wedge N^* & \alpha \wedge N + \beta \wedge \Theta_2 + \beta \wedge L \\
	\gamma \wedge \Theta_1 + \gamma \wedge M - \delta \wedge N^* & \gamma \wedge N + \delta \wedge \Theta_{2} + \delta \wedge L
	\end{pmatrix},
	$$
	so
	\begin{multline}
	\label{AThetaA}
	\Tr(\mathcal{A}  \wedge \Theta_h \wedge \mathcal{A}^*) = \Tr(\alpha \wedge \Theta_{1} \wedge \alpha^*) + \Tr(\gamma \wedge \Theta_{1} \wedge \gamma^*) + \Tr(\alpha \wedge M \wedge \alpha^*) + \Tr(\gamma \wedge M \wedge \gamma^*) \\
	+ \Tr(\alpha \wedge N \wedge \beta^*) - \Tr(\beta \wedge N^* \wedge \alpha^*) \\
	+ \Tr(\gamma \wedge N \wedge \delta^*) - \Tr(\delta \wedge N^* \wedge  \gamma^*) \\
	+ \Tr(\beta \wedge \Theta_{2} \wedge \beta^*) + \Tr(\delta \wedge \Theta_{2} \wedge \delta^*) + \Tr(\beta \wedge L \wedge \beta^*) + \Tr(\delta \wedge L \wedge \delta^*)
	\end{multline}

	Fix a point $ q  \in M $ and write $ \alpha = \alpha_X + \alpha_P $ at the point $ q $, where $ \alpha_X \in \text{End}({V_1}_q) \otimes T^{(1, 0)}_{\pi_1(q)} X $ and $ \alpha_P \in \text{End}({V_1}_q) \otimes T^{(1, 0)}_{\pi_2(q)} \mathbb{P}^1 $. We similarly decompose $ \beta, \gamma, \delta $. Note that $ \Theta_1 \in \text{End}({V_1}_q) \otimes T^{(1, 1)}_{\pi_1(q)} X $, $ M \in \text{End}({V_1}_q) \otimes T^{(1, 1)}_{\pi_2(q)} \mathbb{P}^1 $, $ N \in \text{Hom}({V_2}_q, {V_1}_q) \otimes (T^{(1, 0)}_{\pi_1(q)} X \wedge T^{(0, 1)}_{\pi_2(q)} \mathbb{P}^1) $, etc. Now, at $ q $, \eqref{AATheta} and \eqref{AThetaA} can be written as
	\begin{multline*}
	\Tr(\mathcal{A} \wedge \mathcal{A}^* \wedge \Theta_h) = \Tr((\beta_P \wedge \beta_P^* + \alpha_P \wedge \alpha_P^*) \wedge \Theta_1 ) + \Tr((\beta_X \wedge \beta_X^* + \alpha_X \wedge \alpha_X^*) \wedge M ) \\
	+ \Tr(\gamma_P \wedge \alpha_X^* \wedge N) - \Tr(\alpha_X \wedge \gamma_P^* \wedge N^*) \\
	+ \Tr(\delta_P \wedge \beta_X^* \wedge N) - \Tr(\beta_X \wedge \delta_P^* \wedge N^*) \\
	+ \Tr((\gamma_P \wedge \gamma_P^* + \delta_P \wedge \delta_P^*) \wedge \Theta_2 ) + \Tr((\gamma_X \wedge \gamma_X^* + \delta_X \wedge \delta_X^*) \wedge L ).
	\end{multline*}and
	\begin{multline*}
	\Tr(\mathcal{A}  \wedge \Theta_h \wedge \mathcal{A}^*) = \Tr(\alpha_P \wedge \Theta_{1} \wedge \alpha_P^*) + \Tr(\gamma_P \wedge \Theta_{1} \wedge \gamma_P^*) + \Tr(\alpha_X \wedge M \wedge \alpha_X^*) + \Tr(\gamma_X \wedge M \wedge \gamma_X^*) \\
	+ \Tr(\alpha_P \wedge N \wedge \beta_X^*) - \Tr(\beta_X \wedge N^* \wedge \alpha_P^*) \\
	+ \Tr(\gamma_P \wedge N \wedge \delta_X^*) - \Tr(\delta_X \wedge N^* \wedge  \gamma_P^*) \\
	+ \Tr(\beta_P \wedge \Theta_{2} \wedge \beta_P^*) + \Tr(\delta_P \wedge \Theta_{2} \wedge \delta_P^*) + \Tr(\beta_X \wedge L \wedge \beta_X^*) + \Tr(\delta_X \wedge L \wedge \delta_X^*)
	\end{multline*}
	Notice that in the sum $ \Tr(\mathcal{A} \wedge \mathcal{A}^* \wedge \Theta_h) + \Tr(\mathcal{A}  \wedge \Theta_h \wedge \mathcal{A}^*) $, the sets of variables $ \{ \alpha_X, \gamma_P, \delta_X \} $, $ \{  \alpha_P, \beta_X, \delta_P \} $, $ \{ \beta_P \} $, and $ \{ \gamma_X \} $ are mutually decoupled. Let
	\begin{multline*}
	\mathfrak{B}_1 = \Tr( i\alpha_X \wedge \alpha_X^* \wedge iM ) + \Tr(i\alpha_X \wedge iM \wedge \alpha_X^*) + \Tr(i\gamma_P \wedge i\Theta_{1} \wedge \gamma_P^*) \\ + \Tr(i\delta_X \wedge iL \wedge \delta_X^*) + \Tr(i\delta_X \wedge \delta_X^* \wedge iL ) + \Tr(i\gamma_P \wedge \gamma_P^* \wedge i\Theta_{2}) \\
	+ \Tr(i\gamma_P \wedge iN \wedge \delta_X^*) - \Tr(i\delta_X i\wedge N^* \wedge  \gamma_P^*) + \Tr(i\gamma_P \wedge \alpha_X^* \wedge iN) - \Tr(i\alpha_X \wedge \gamma_P^* \wedge iN^*)	
	\end{multline*}
	\begin{multline*}
	\mathfrak{B}_2 = \Tr(i\alpha_P \wedge \alpha_P^* \wedge i\Theta_1 ) + \Tr(i\alpha_P \wedge i\Theta_{1} \wedge \alpha_P^*) + \Tr(i\beta_X \wedge \beta_X^* \wedge iM ) \\
	+ \Tr( i\delta_P \wedge \delta_P^* \wedge i\Theta_2 ) + \Tr(i\delta_P \wedge i\Theta_{2} \wedge \delta_P^*) + \Tr(i\beta_X \wedge iL \wedge \beta_X^*) \\
	+ \Tr(i\delta_P \wedge \beta_X^* \wedge iN) - \Tr(i\beta_X \wedge \delta_P^* \wedge iN^*) + \Tr(i\alpha_P \wedge iN \wedge \beta_X^*) - \Tr(i\beta_X \wedge iN^* \wedge \alpha_P^*)
	\end{multline*}		
	\begin{align*}
	\mathfrak{B}_3 =  \Tr(i\beta_P \wedge i\Theta_{2} \wedge \beta_P^*) + \Tr(i\beta_P \wedge \beta_P^* \wedge i\Theta_1 )
	\end{align*}
	\begin{align*}	
	\mathfrak{B}_4 = \Tr(i\gamma_X \wedge iM \wedge \gamma_X^*) + \Tr(i\gamma_X \wedge \gamma_X^* \wedge iL )
	\end{align*}
	It follows that $ i\Theta_h $ is MA-(semi)positive at $ q $ iff each $ \mathfrak{B}_i $ is (semi)positive. \\
	
	\begin{rem*}
		For all $ SU(2)$-invariant connections $ \nabla_h $ of the sort we consider in this section, if $ g $ is an $ SU(2) $-invariant endomorphism, the only non-vanishing components of $ \nabla_h $ at any point are $ \alpha_X $, $ \gamma_P $ and $ \delta_X $, so restricted MA-positivity in this case corresponds to positivity of $ \mathfrak{B}_1 $.
	\end{rem*}
	
	Consider the bilinear form $ \mathfrak{B}_1 $. Let $ w $ be a local holomorphic coordinate for $ X $ near $ \pi_1(q) $, and $ z $ the standard holomorphic coordinate for $ \mathbb{P}^1 $ near $ \pi_2(q) $. By an abuse of notation, let us recycle the variables $ \alpha, \gamma, \delta  $ and write $ \alpha_X =: \alpha \cdot dw  $,  $ \delta_X =: \delta\cdot dw  $, $ \gamma_P = (\gamma \otimes \frac{\partial}{\partial z}) dz $, with $ \alpha \in \text{End}({E_1}_{\pi_1(q)}) $, $ \delta \in \text{End}({E_2}_{\pi_2(q)}) $, $ \gamma \in \text{Hom}({E_1}_{\pi_1(q)}, {E_2}_{\pi_2(q)}) $. Also, let $ A \cdot idw \wedge d \bar{w} := i\Theta_{1} $, $ A' \cdot idw \wedge d \bar{w} := i\Theta_{2} $, $ B \omega_{FS} = iM $, $ B' \omega_{FS} := iL $, $ (C \otimes \frac{dz}{(1 + |z|^2)^2}) dw \wedge d \bar{z} = N $. Let $ n $ be the rank of $ E_1 $. From here on, we take $ \text{rank}(E_2) = 1 $ and suppose that orthonormal bases are chosen for $ E_1, E_2 $. With such a choice of basis, we have (after dividing $ \mathfrak{B}_1 $ by $ idw \wedge d\bar{w} \wedge idz \wedge d\bar{z} $)
	\begin{align*}
	\mathfrak{B}_1 = \Tr(\{B, \alpha \}\alpha^*) + \gamma A \gamma^* + 2B' |\delta|^2 + A' |\gamma|^2 - 2 \Re(\bar{\delta} \gamma C ) - 2 \Re( \gamma \alpha^* C ),
	\end{align*} and with an analogous notation (with $ \alpha, \beta, \delta $ now re-defined in terms of $ \alpha_P, \beta_X, \delta_P $ instead),
	\begin{align*}
	\mathfrak{B}_2 = \Tr(\{A, \alpha \}\alpha^*) + \beta^* B \beta + 2A' |\delta|^2 + B' |\beta|^2 - 2 \Re(\bar{\delta} \beta C^* ) - 2 \Re( \beta^* \alpha C ).
	\end{align*}In this notation, the vbMA equation can be written as
	\begin{align}
	\label{vortexvbMAsurface}
	\{A, B\} - CC^* = k \cdot \text{Id}_n \\
	\label{vortexvbMAsurface2}
	\{A', B'\} - C^*C = k
	\end{align}for some $ k > 0 $.
	Suppose that the orthonormal bases of $ E_1 $ and $ E_2 $ have also be chosen such that $ B = 2r + TT^* $ is diagonal. We then have
	\begin{align*}
	\mathfrak{B}_1 = \sum_{i, j} |\alpha_{i, j}|^2 (b_i + b_j) + \sum_{i, j} a_{ij}\gamma_i \bar{\gamma_j} + 2B' |\delta|^2 + A' |\gamma|^2 - 2 \Re(\bar{\delta} \sum_i \gamma_i c_i ) - 2 \Re( \sum_{i, j} \bar{\alpha_{j, i}} c_i \gamma_j ),
	\end{align*} where the $ b_i $s are the entries of the diagonalized matrix $ B $, $ a_{ij} $ are the entries of $ A $, $ c_i $ are the entries of $ C $, and so on.
	
	 As shown in the proof of the lemma below, we have $ B' > 0 $, so \eqref{vortexvbMAsurface}, \eqref{vortexvbMAsurface2} can be solved for $ A $ and $ A' $ in terms of $ k $, $ B $, $ B' $ and $ C $:
	\begin{equation}
		\label{equationsfora}
		a_{ij} = \frac{k \delta_{ij} + c_i \bar{c_j}}{b_i + b_j}, A' = \frac{k + |C|^2}{2B'}.
	\end{equation}
	
	\begin{rem*}
		\label{bivalues}
		Since $ TT^* $ has rank at most 1 when $ \text{rank}(E_2) = 1 $, we can have $ b_i \neq 2r $ for at most a single $ i \in {1, \dots, n} $, which we take to be $ i = 1 $. We then have $ b_1 = 2r + |T|_h^2 $ and $ B' = 2r + 2 - |T|_h^2 $. Since $ B $ is supposed to be diagonal in the chosen basis, we must also have $ T = (t, 0 \dots, 0) $ in this basis for some $ t \in \mathbb{C} $. Thus, $ |T|_h $, the entries of $ C dw = \nabla_h^{1, 0} T $ and $ k $ are the only independent parameters involved here.
	\end{rem*}

	\begin{lem}
		\label{boundOnT}
	If the metric $ h $ is a solution of the vbMA equation, then $ |T|_{h} \leq 1 $
\end{lem}
\begin{proof}
	By the semi-positivity of $ \mathfrak{B}_2 $, $ A' \geq 0 $. By $ \eqref{vortexvbMAsurface2} $ and the fact that $ k > 0 $, $ A' B' > 0 $ and hence $ A', B' > 0 $. Differentiating $ \langle T, T \rangle_h $ twice gives $ \partial \bar{\partial} \langle T, T \rangle_h = -D^{0, 1}T^{*_h} \wedge D_h^{1, 0} T - T^{*_h} \Theta_{1} T + \Theta_{2} |T|_h^2 $. At a point $ p $ where $ |T|^2_h $ attains its maximum, $ i \partial \bar{\partial} \langle T, T \rangle_h \leq 0 $ and it is always true that $ iD^{0, 1}T^{*_h} \wedge D_h^{1, 0} T \leq 0$, so it follows that at $ p $, $-T^{*_h} i\Theta_{1} T + i\Theta_{2} |T|_h^2 \leq 0 $. By the above remark and the expressions for $ A, A' $ (i.e. $ \Theta_{1} $, $ \Theta_{2} $), this is equivalent to the inequality
	$$ -\frac{(k + |c_1|^2)|T|_h^2}{2(2r + |T|_h^2)} + \frac{(k + |C|^2)|T|_h^2}{2(2r + 2 - |T|_h^2)} \leq 0.
	$$So, at $ p $, we must have $$ 0 \geq -\frac{(k + |c_1|^2)|T|_h^2}{2(2r + |T|_h^2)} + \frac{(k + |c_1|^2)|T|_h^2}{2(2r + 2 - |T|_h^2)} = \frac{(k + |c_1|^2)|T|_h^2(|T|_h^2 - 1)}{(2r + |T|_h^2)(2r + 2 - |T|_h^2)}   .$$Since $ T $ is not identically zero and $ 2r + 2 - |T|_h^2 > 0 $, this implies $ \sup |T|_h^2 = |T|_h^2(p) \leq 1 $.
\end{proof}

\begin{thm}
	There exist solutions to the algebraic equations \eqref{vortexvbMAsurface}, \eqref{vortexvbMAsurface2} which are MA-semi-positive but not MA-positive.
\end{thm}
\begin{proof}
	See (counter-)example \ref{exampleOnsurfaces} below.
\end{proof}

	To gain more insight into the nature of the constraints imposed by the equations \eqref{vortexvbMAsurface}, \eqref{vortexvbMAsurface2} combined with MA-semi-positivity and to also find a procedure for coming up with such counterexamples, we find precise conditions for the positivity of $ \mathfrak{B}_1 $. As $ \mathfrak{B}_1 $ is a bilinear form in the matrices $ \alpha, \gamma, \delta $, the (semi)-positivity of $ \mathfrak{B}_1 $ is equivalent to the (semi)-positivity of the $ (n^2 + n + 1) \times (n^2 + n + 1) $ Hermitian matrix
	$$\mathcal{M} = \begin{pmatrix}
		\mathcal{B} & -\mathcal{D}^* & 0 \\
		-\mathcal{D} & A + A' \cdot \text{Id}_r & -C \\
		0 & -C^* & 2B'
	\end{pmatrix},
	$$ where $ \mathcal{B} $ is the $ n^2 \times n^2 $ diagonal matrix with entries $ \mathcal{B}_{ij, i'j'} = (b_i + b_j) \delta_{ii'} \delta_{jj'} $ and $ \mathcal{D} $ is the $ n \times n^2 $ matrix with entries $ \mathcal{D}_{j', ij} = \bar{c_j} \delta_{j'i}  $. \\
	
	Since $ B $ is positive definite, so is $ \mathcal{B} $. Hence, $ \mathcal{M} $ is positive-(semi)definite iff the Schur complement
	$$
	\mathcal{P} = \begin{pmatrix}
	A + A' \cdot \text{Id}_n & -C \\
	-C^* & 2B'
	\end{pmatrix} - \begin{pmatrix}
	-\mathcal{D} \\ 0
	\end{pmatrix} \mathcal{B}^{-1} \begin{pmatrix}
	-\mathcal{D}^* & 0
	\end{pmatrix}
	$$is positive-(semi)definite. It is easy to check that the $ (n + 1) \times (n + 1) $ matrix $$ \begin{pmatrix}
	-\mathcal{D} \\ 0
	\end{pmatrix} \mathcal{B}^{-1} \begin{pmatrix}
	-\mathcal{D}^* & 0
	\end{pmatrix} = \begin{pmatrix}
	\mathcal{Q}& 0 \\ 0 & 0
	\end{pmatrix},
	$$ where $ \mathcal{Q} $ is the $ n \times n $ diagonal matrix with entries $ \displaystyle \mathcal{Q}_{ij} = \delta_{ij} \sum_{k = 1, \dots, r} \frac{|c_k|^2}{b_i + b_k} $. Hence, we get
	$$
	\mathcal{P} = \begin{pmatrix}
	A + A' \cdot \text{Id}_n - \mathcal{Q} & -C \\
	-C^* & 2B'
	\end{pmatrix}.
	$$
	Furthermore, since we know that $ B' > 0 $, $ \mathcal{P} $ is positive-(semi)definite iff the Schur complement
	$$
	\label{mainMApos}
	\mathcal{R} = A + A' \cdot \text{Id}_n - \mathcal{Q} - \frac{CC^*}{2B'}
	$$is positive-(semi)definite. By the earlier remark that $ b_i = b_j = 2r $ for all $ i, j > 1 $ and the expression $ \eqref{equationsfora} $ for $ A $, we can write $ \mathcal{R} $ as
	$$
	\mathcal{R} = \begin{pmatrix}
	p' & \bar{q'} v^* \\
	q' v & r' \cdot \text{Id}_{n - 1} + s' \cdot vv^*
	\end{pmatrix},
	$$where $ v $ is the column vector $  (c_2, \dots, c_n)^t $, $  p' = \frac{k}{2b_1} + A' - \frac{|v|^2}{b_1 + 2r} - \frac{|c_1|^2}{2B'} $, $  q' = \bar{c_1} (\frac{1}{b_1 + 2r} - \frac{1}{2B'}) $, $ r' = \frac{k}{4r} + A' - \frac{|v|^2}{4r} + \frac{|c_1|^2}{b_1 + 2r} $ and $ s' = (\frac{1}{4r} - \frac{1}{2B'}) $. \\
	
	From this we see that $ \mathcal{R} $ is positive definite iff
	\begin{align*}
		p' > 0 \text{ and }  r' \cdot \text{Id}_{n - 1} + \bigg(s' - \frac{|q'|^2}{p'} \bigg) vv^* > 0
	\end{align*}(note that the eigenvalues of this latter matrix are just $ r' $ and $ \displaystyle r' + |v|^2(s' - \frac{|q'|^2}{p'}) $). Also, $ \mathcal{R} $ is strictly semi-definite iff either
	\begin{align}
	\label{semidefpossibility1}
		 p' = 0, q'v = 0 , \text{ and } r' \cdot \text{Id}_{n - 1} + s' \cdot vv^* \geq 0
	\end{align}or
	\begin{align}
	\label{semidefpossibility2}
		  p' > 0 \text{ and }  r' \cdot \text{Id}_{n - 1} + \bigg(s' - \frac{|q'|^2}{p'} \bigg) vv^* \text{ is strictly semi-definite}.
	\end{align}Thus, to show that strict semi-definiteness is a possibility, we need to produce matrices which satisfy one of the conditions for strict $ MA $-semi-positiveness while also solving the vbMA equation. In other words, we need to show that the independent parameters here (see the previous remark) are not overdetermined by these conditions. We only consider the case \eqref{semidefpossibility1} for simplicity. \\
	
	In this case, by the definition of $ p' $ and \eqref{equationsfora}, $ p' = 0 $ is equivalent to the equation $ k(\frac{1}{2B'} + \frac{1}{2b_1}) = |v|^2 (\frac{1}{b_1 + 2r} - \frac{1}{2B'})  $. Hence, $ v \neq 0 $. Since $ q' = 0 $ is required for semi-positivity in this case and since $ q = v(\frac{1}{b_1 + 2r} - \frac{1}{2B'})\bar{c_1} $, we get that $ c_1 = 0 $, which implies $ |v|^2 = |C|^2 $. When $ c_1 = 0 $, the only eigenvalue of $ r' \cdot \text{Id}_{n - 1} + s' \cdot vv^* \geq 0  $ which is not necessarily positive is $ r' $ as can be checked by using the expressions for $ r', s',$ etc. Using $ p' = 0 $, $ c_1 = 0 $, we get $ r' = |v|^2(\frac{1}{b_1 + 2r} - \frac{1}{4r}) $. Since $ b_1 = 2r + |T|^2 \geq 2r$, we see that $ r' \leq 0 $ and $ r' = 0 $ iff $ b_1 = 2r $ or equivalently, $ |T|^2 = 0 $ and consequently $ 2B' = 4r + 4 $. So the conditions $ b_1 = 2r $, and $ |C|^2 = k(2r + 1) $ are necessary for the first case. \\
	
	Conversely, if $ b_1 = 2r $ (equivalently, $ |T|_h = 0 $), the matrix $ B $ is a multiple of the identity and hence we can choose orthonormal bases such that $ C $ is of the form $ (0, \dots, c) $ with $ c \in \mathbb{C} $. Hence, $ c_1 = 0 $ and $ |v|^2 = |C|^2 $ in this basis. Since $ c_1 = 0 $, we get $ q = 0, r' = k(\frac{1}{4r} + \frac{1}{4r + 4}) + |C|^2(\frac{1}{4r + 4} - \frac{1}{4r}) = p'$. If, in addition, $|C|^2 = k(2r + 1)$, then $ r' = p' = 0 $ and strict semi-positivity follows. In other words, case \eqref{semidefpossibility1} occurs precisely on the level set defined by the equations $ T = 0$ and $-iD^{0, 1}T^* \wedge D^{1, 0}T = \eta (2r + 1) $. The second case can be characterized in a similar manner. Plugging in a convenient value for $ k $, we get the following example of strict MA-semi-positivity, which we also verify directly.
 	
	\begin{eg}
	\label{exampleOnsurfaces}
	Take $ k = 4, r = 1$, $ B = 2 \cdot \text{Id} $, $ B' = 4$, $ A' = 2 $,
	$
	A =
	\begin{pmatrix}
	1 & 0 & \dots & 0 \\
	0 & 1 & \dots & 0 \\
	\vdots & \vdots & \ddots & \vdots \\
	0 & 0 & \dots & 4
	\end{pmatrix}
	$ and $
	C =
	\begin{pmatrix}
	0 \\
	\vdots \\
	2 \sqrt{3}
	\end{pmatrix}
	$
	\begin{prop}
		For this choice of matrices and parameters, the vbMA equation is satisfied and the resulting matrices are MA-semi-positive but not MA-positive
	\end{prop}
\begin{proof}
	That these matrices satisfy the vbMA equations \eqref{vortexvbMAsurface}, \eqref{vortexvbMAsurface2} is simple to check. To prove the remaining part of the proposition, we write down
	\begin{align*}
	\mathfrak{B}_1 = 4 \sum_{i, j} |\alpha_{i, j}|^2 + 3|\gamma|^2 + 3 |\gamma_n|^2 + 8 |\delta|^2 - 4 \sqrt{3} \Re(\bar{\delta} \gamma_n ) - 4 \sqrt{3} \Re( \sum_{j} \bar{\alpha_{n, j}} \gamma_j ) \\
	=  4 \sum_{i < n, j} |\alpha_{i, j}|^2 + \sum_j |2 \alpha_{n, j} - \sqrt{3} \gamma_j|^2 + |2 \delta - \sqrt{3} \gamma_n|^2 + 4 |\delta|^2
	\end{align*}
	which is clearly positive semi-definite but not definite for $ n > 1 $. Also, for these matrices,
	\begin{align*}
	\mathfrak{B}_2 = 2 \sum_{i < n, j < n} |\alpha_{i, j}|^2 + 5\sum_{i < n} (|\alpha_{i, n }|^2 + |\alpha_{n, i}|^2) + 8 |\alpha_{n,n}|^2 +  6|\beta|^2 + 4 |\delta|^2 \\
	- 4 \sqrt{3} \Re(\bar{\delta} \beta_n ) - 4 \sqrt{3} \Re( \sum_{j} \bar{\alpha_{j, n}} \beta_j )  \\
	= 2 \sum_{i < n, j < n} |\alpha_{i, j}|^2 + 5\sum_{i < n}  |\alpha_{n, i}|^2 + \sum_{i < n} (5|\alpha_{i, n }|^2 + 6|\beta_i|^2 - 4 \sqrt{3}\Re(\bar{\alpha_{i, n}} \beta_i)) + \\
	(8 |\alpha_{n,n}|^2 +  2|\beta_n|^2 - 4 \sqrt{3}\Re(\bar{\alpha_{n, n}} \beta_n)) + 4(|\beta_n|^2 +  |\delta|^2 -  \sqrt{3} \Re(\bar{\delta} \beta_n ))
	\end{align*}
	which is seen to be positive definite by completing squares. It is also easy to see that $ \mathfrak{B}_3, \mathfrak{B}_4 $ must be positive definite. Hence, we have here an example of matrices satisfying the algebraic vbMA equation and MA-semi-positivity but not MA-positivity.
\end{proof}
\end{eg}

\begin{thm}
	For $ d \geq 2, r \geq 3 $, it is not necessarily true that $ S(d, r) \cap V(d, r, \eta) \subseteq P(d, r) $ (see \ref{sec:MApositivity} for the definitions of these sets).
\end{thm}
\begin{proof}
The above example proves the $ d = 2 $ version of this theorem. For $ m := d - 2 > 0, n \geq 2 $, we use this example to produce strictly MA-semi-positive, $\text{End}(\mathbb{C}^{n+1})$-valued (1, 1)-forms on $ \mathbb{C}^2 \times \mathbb{C}^m $ satisfying an algebraic vbMA equation. Let $ w, z $ be coordinates on $ \mathbb{C}^2 $ as before and let $ u^i, i = 1, \dots, m $ be coordinates on $ \mathbb{C}^m $. Define $ i\Phi = i\Theta + i\Psi $, where $ \displaystyle i \Psi = i(\sum_j du^j \wedge d \bar{u}^j) \cdot \text{Id}_{n+1} $ and $ i \Theta $ is defined by the matrices in the above example:
\begin{equation}
\label{Thetaexp}
	i \Theta = \begin{pmatrix}
	A & 0 \\ 0 & A'
	\end{pmatrix} i dw \wedge d\bar{w} +
	\begin{pmatrix}
	B & 0 \\ 0 & B'
	\end{pmatrix} \omega_{FS} +
	\begin{pmatrix}
	0 & iN \\  -iN^* & 0
	\end{pmatrix}
\end{equation}with $ A, A', B, B', C, k, r $ being as in the example, $ N = C \otimes \frac{dz}{(1 + |z|^2)^2} \cdot  dw \wedge d \bar{z} $, $ N^* = C^* \otimes \frac{\partial}{\partial z} \cdot d\bar{w} \wedge d z  $ and $ \omega_{FS} = \frac{idz \wedge d \bar{z}}{(1 + |z|^2)^2} $ as earlier. We shall show that $ i \Phi $ solves an algebraic vbMA equation and is MA-semi-positive but not MA-positive. \\

Since $ i \Psi  $ commutes with $ i \Theta $, we have $ (i \Phi)^\mu = \sum_{\nu = 0}^{\mu} { \mu \choose \nu } (i \Theta)^\nu (i \Psi)^{\mu - \nu} $ for any positive integer $ \mu $. Also, since $ (i \Theta)^\mu = 0  $ and $ (i \Psi)^\nu = 0 $ for $ \mu > 2, \nu > m $, it follows that
\begin{multline}
	 (i \Phi)^{m + 2} = { m + 2 \choose 2} (i \Theta)^2 (i \Psi)^m \\ = k \frac{(m + 2)!}{2!} \omega_{FS} \wedge i dw \wedge d\bar{w} \wedge i du^1 \wedge d \bar{u}^1 \wedge \dots i\wedge du^m \wedge d \bar{u}^m \cdot \text{Id}_{n+1}
\end{multline} where $ k $ is the constant in the above example. Hence, $ i \Phi $ satisfies the algebraic vbMA equation. It remains to show that $ i \Phi $ is strictly MA-semi-positive. \\

Let $ \mathcal{A} $ be an $ \text{End}(\mathbb{C}^{n+1}) $-valued $ (1, 0) $-form on $ \mathbb{C}^{m + 2} $. We have
\begin{multline*}
	\sum_{\mu = 0}^{m+1} \Tr ( (i \Phi)^{\mu} \wedge i \mathcal{A} \wedge (i \Phi)^{m + 1 - \mu} \wedge \mathcal{A}^{*} ) \\
	= \sum_{\mu = 0}^{m+1} \sum_{\nu = 0}^{\mu} \sum_{\xi = 0}^{m + 1 - \mu} \Tr ( {\mu \choose \nu }(i \Theta)^{\nu} (i \Psi)^{\mu - \nu} \wedge i \mathcal{A} \wedge {m + 1 - \mu \choose \xi } (i \Theta)^{\xi} (i \Psi)^{m + 1 - \mu - \xi} \wedge \mathcal{A}^{*} ) \\
	= \sum_{\substack{0 \leq \nu, \xi \\ \nu + \xi \leq m + 1}} \sum_{\mu = \nu}^{m + 1 - \xi} \Tr ( (i \Psi)^{m + 1 - \nu - \xi} {\mu \choose \nu } {m + 1 - \mu \choose \xi } (i \Theta)^{\nu} \wedge i \mathcal{A} \wedge  (i \Theta)^{\xi} \wedge \mathcal{A}^{*} ) \\
	= \sum_{\substack{0 \leq \nu, \xi \\ \nu + \xi \leq m + 1}} \Tr ( (i \Psi)^{m + 1 - \nu - \xi} {m + 2 \choose \nu + \xi + 1 } (i \Theta)^{\nu} \wedge i \mathcal{A} \wedge  (i \Theta)^{\xi} \wedge \mathcal{A}^{*} ) \\
	= \Tr((i \Psi)^m {m + 2 \choose 2} (i \Theta  \wedge i\mathcal{A} \wedge \mathcal{A}^{*} + i \mathcal{A} \wedge i \Theta \wedge \mathcal{A}^{*}) ) \\ +
	\Tr((i \Psi)^{m - 1} {m + 2 \choose 3} ((i \Theta)^2  \wedge i\mathcal{A} \wedge \mathcal{A}^{*} + i \mathcal{A} \wedge (i \Theta)^2 \wedge \mathcal{A}^{*} + i \Theta \wedge i \mathcal{A} \wedge i \Theta \wedge \mathcal{A}^{*})),
\end{multline*}
where the combinatorial identity used in the fourth line follows from equating the coefficients of $ x^\nu y^\xi $ on the two ends of the identity $ \displaystyle \sum_{\mu = 0}^{m+1}(1 + x)^\mu(1 + y)^{m + 1 - \mu} = \frac{(1 + x)^{m+2} - (1 + y)^{m+2}}{x - y} = \sum_{k = 1}^{m+2} {m + 2 \choose k} \sum_{j = 0}^{k-1} x^{k - 1 - j}y^j$.\\

 Writing $ \mathcal{A} = \mathcal{B} + \mathcal{C} $, where $  \mathcal{B} $ is the component of $ \mathcal{A} $ along $ \mathbb{C}^{2} $ and $ \mathcal{C} $ is the component along $ \mathbb{C}^{m} $,  the above expression can be written as
\begin{multline}
\label{higherdimMApos}
	\sum_{\mu = 0}^{m+1} \Tr ( (i \Phi)^{\mu} \wedge i \mathcal{A} \wedge (i \Phi)^{m + 1 - \mu} \wedge \mathcal{A}^{*} ) \\
	= \Tr((i \Psi)^m {m + 2 \choose 2} (i \Theta  \wedge i\mathcal{B} \wedge \mathcal{B}^{*} + i \mathcal{B} \wedge i \Theta \wedge \mathcal{B}^{*}) ) \\ +
	\Tr((i \Psi)^{m - 1} {m + 2 \choose 3} ((i \Theta)^2  \wedge i\mathcal{C} \wedge \mathcal{C}^{*} + i \mathcal{C} \wedge (i \Theta)^2 \wedge \mathcal{C}^{*} + i \Theta \wedge i \mathcal{C} \wedge i \Theta \wedge \mathcal{C}^{*})).
\end{multline}
Since $ \mathcal{B} $ and $ \mathcal{C} $ are independent of each other, it follows that $ i\Phi $ is MA-(semi-)positive iff each of the two terms in the right-hand-side of \eqref{higherdimMApos} are (semi-)positive for arbitrary $ \mathcal{B} $, $ \mathcal{C} $. As $ i \Theta $ is strictly MA-semi-positive, the first term in \eqref{higherdimMApos} is strictly semi-positive as well and hence $ i \Phi $ can at most be strictly MA-semi-positive. Hence, to show the strict MA-semi-positivity of $ i \Phi $, it suffices to show that the second term is semi-positive. \\

Decomposing $ \mathcal{C} $ further as $ \mathcal{C} = \sum_i \mathcal{D}_i du^i $, with each $ \mathcal{D}_i \in \text{End}(\mathbb{C}^{n+1}) $ and using the fact that $ i \Psi = i \sum_j du^j \wedge d \bar{u}^j \cdot \text{Id}_{n+1} $, we see that it is sufficient to prove that
\begin{equation}
\label{sufficientpos}
	\Tr((i \Theta)^2 \mathcal{D} \mathcal{D}^{*} + \mathcal{D} (i \Theta)^2 \mathcal{D}^{*} + i \Theta \mathcal{D} \wedge i \Theta \mathcal{D}^{*}) \geq 0
\end{equation}for each $ \mathcal{D} \in \text{End}(\mathbb{C}^{n+1}) $. Since $ (i \Theta)^2 $ is a positive multiple of $ \text{Id}_{n+1} $, the first two terms in \eqref{sufficientpos} are positive for $ \mathcal{D} \neq 0 $. Let $ \mathcal{D}  =: \begin{pmatrix}
\alpha & \beta \\ \gamma & \delta
\end{pmatrix} \in \text{End}(\mathbb{C}^{n} \times \mathbb{C}) $. Writing the expression \eqref{Thetaexp} for $ i \Theta $ as $ i \Theta =: i \Theta_w + i \Theta_z + i \Theta_{wz}  $, it follows that
\begin{equation*}
	i \Theta \mathcal{D} \wedge i \Theta \mathcal{D}^{*} = i \Theta_w \mathcal{D} \wedge i \Theta_{z} \mathcal{D}^{*} + i \Theta_z \mathcal{D} \wedge i \Theta_{w} \mathcal{D}^{*} + i \Theta_{wz} \mathcal{D} \wedge i \Theta_{wz} \mathcal{D}^{*}.
\end{equation*}
Now, a straightforward computation gives
\begin{align*}
\Tr(i \Theta \mathcal{D} \wedge i \Theta \mathcal{D}^{*})/(i dw \wedge d\bar{w} \wedge \omega_{FS}) = \Tr(A \alpha B \alpha^*) + \Tr(A \beta B' \beta^*) + A' \gamma B \gamma^* + A'B' |\delta|^2 \\
+ \Tr(B \alpha A \alpha^*) + \Tr(B \beta A' \beta^*) + B' \gamma A \gamma^* + A'B' |\delta|^2 \\
- 2 \Re(\delta C^* \alpha^* C).
\end{align*}

It is easy to see that all terms involving $ \beta, \gamma $ are positive since $ A, B, A', B' $ are all positive definite (for example, $\Tr(A \beta B' \beta^*) = \Tr(\sqrt{A} \beta B' \beta^* \sqrt{A}) \geq 0 $). Since $ C = (0, \dots, 2 \sqrt{3})^t $ has only one non-zero entry, $ 2 \Re(\delta C^* \alpha^* C) = 2 \Re(\delta (2 \sqrt{3})^2 \bar{\kappa} ) $, where $ \kappa $ is defined by $ \alpha =: \begin{pmatrix}
* & * \\ * & \kappa
\end{pmatrix} \in \text{End}(\mathbb{C}^{n-1} \times \mathbb{C}) $. As $ A $ is diagonal and $ B $ is a multiple of the identity matrix $ \text{Id}_n $, it is straightforward to show that $ \displaystyle \Tr(A \alpha B \alpha^*) + 2 A'B' |\delta|^2 + \Tr(B \alpha A \alpha^*) = 16(|\delta|^{2} + |\kappa|^2) $ + positive terms not involving $\kappa$ or $\delta$. Hence,
\begin{multline*}
	\displaystyle \Tr(i \Theta \mathcal{D} \wedge i \Theta \mathcal{D}^{*})/(i dw \wedge d\bar{w} \wedge \omega_{FS}) = 16(|\delta|^{2} + |\kappa|^2) - 24 \Re(\delta \bar{\kappa} ) \\
	 + \text{ (positive terms not involving $\kappa$ or $\delta$)}.
\end{multline*}  By completing squares, it follows that this expression is positive-definite and hence \eqref{sufficientpos} $ \geq 0 $, proving that $ i \Phi $ is strictly MA-semi-positive.
\end{proof}

	\section{\textbf{Vortex bundles over three-folds}}\label{sec:threefolds}
	In this final section, we show that MA-semi-positive solutions of the vbMA equation on a class of rank-2 vortex bundles \cite{Gar94}, \cite{Pin18} over complex three-folds are MA-positive in a restricted sense. Consider a K\"ahler manifold $ X $ of dimension $ n $ and a positive holomorphic line bundle $ L $ over $ X $. Let $ M = X \times  \mathbb{CP}^1  $ and let $ \pi_1 $ and $ \pi_2 $ be the projections onto $ X $ and $  \mathbb{CP}^1  $, respectively.  Let $ r_1 $ and $ r_2 $ be positive integers and define, as in the previous section,
	
	\begin{defn*}
		As a $ C^\infty $ bundle, the vortex bundle $ V $ is defined as $ V = V_1 \oplus V_2 $ where $$ V_1 = (r_1 + 1)\pi_1^{*}L \otimes r_2 \pi_2^{*}\mathcal{O}(2) $$ $$ V_2 = r_1\pi_1^{*}L \otimes (r_2 + 1) \pi_2^{*}\mathcal{O}(2)  .$$
	\end{defn*}
	
	As shown in \cite{Gar94}, all $ SU(2) $-invariant metrics on $ V $ are of the form $ (h_0^{r_1 + 1} e^{-\phi_1} \otimes h_{FS}^{r_2}) \oplus (h_0^{r_1} e^{-\phi_2} \otimes h_{FS}^{(r_2 + 1)}) = h_1 \oplus h_2 $ where $ h_0 $ is some metric on $ L $, $ h_{FS} $ is the Fubini-Study metric as in the previous section and $ \phi_1, \phi_2 $ are smooth real functions on $ X $. \\
	
	 For each such metric, we define an integrable connection on $ V $ as before
	
	$$ D_h =
	\begin{pmatrix}
	D_1 & \beta \\
	-\beta^{*} & D_2
	\end{pmatrix}$$
	
where $ D_i $ is the Chern connection on $ V_i $ for the metric $ h_i $ and
\begin{gather}
 \displaystyle \beta = \pi_1^{*} \phi \otimes \pi_2^{*}\left(\frac{dz}{(1 + |z|^2)^2} \otimes d\bar{z}\right) \nonumber   \\
\in \Omega^{0, 1}(M, \text{Hom}(V_2, V_1)) = \Omega^{0, 1}(M, \text{Hom}(\pi_2^{*}\mathcal{O}(2), \pi_1^{*}L)). \nonumber
\end{gather}
 The adjoint $ \beta^{*}  $ is taken using the induced metric $\text{Hom}(V_2, V_1)$, which in this case is $ h_L = h_0 e^{\phi_2 - \phi_1} $ on $ \pi_1^*L $ and $ h_{FS} $ on $ \pi_2^*\mathcal{O}(2) $. As in the previous section, it can be checked that $ D_{h}^{0, 1} \wedge D_{h}^{0, 1} = 0 $. Hence, this connection defines a holomorphic structure on $ V $ for which it is the Chern connection and this holomorphic structure is also independent of the choice of invariant metric $ h $. \\
	
	The curvature of $ D_h $ is
	$$  \Theta_h =
	\begin{pmatrix}
	\Theta_{h_1} - \beta \wedge \beta^* & D^{1, 0} \beta \\
	-D^{0, 1} \beta^{*} & \Theta_{h_2} - \beta^* \wedge \beta
	\end{pmatrix} $$ $$ =	\begin{pmatrix}
	\Theta_1 - i (2r_2 + |\phi|_L^2 )\omega_{FS} & D^{1, 0} \beta \\
	-D^{0, 1} \beta^{*} & \Theta_2 - i(2r_2 + 2 - |\phi|_L^2) \omega_{FS}
	\end{pmatrix} $$ where $ \Theta_1 = (r_1 + 1) \Theta_{h_0} + \partial \bar{\partial} \phi_1 $ and $ \Theta_2 = r_1 \Theta_{h_0} + \partial \bar{\partial} \phi_2 $ and $|\phi|_L $ refers to the norm of $ \phi $ with respect to the metric $ h_L $ on $ L $. \\
	
	We note down an expression for the $k^{th}$-powers of $ \Theta_h $ using induction on $k$.
	\begin{lem}
		\label{ThetaPowers}
	For $ 1 \leq k \leq n + 1 $
	$$ \Theta_h^{k} =
	\begin{pmatrix}
	\Theta_1^k -iP_{k, 1}(\Theta_1, \Theta_2 ) \wedge \omega_{FS} & Q_k(\Theta_1, \Theta_2) \wedge D^{1, 0} \beta \\
	-Q_k(\Theta_1, \Theta_2) \wedge D^{0, 1} \beta^{*} & \Theta_2^{k} -iP_{k, 2}(\Theta_1, \Theta_2) \wedge \omega_{FS}
	\end{pmatrix} $$where \\
	$ P_{k, 1}(x, y) = kax^{k - 1} + G \sum_{j = 0}^{k-2} (j + 1)x^j y^{k - 2 - j} $, \\
	$ P_{k, 2}(x, y) = kby^{k - 1} + G \sum_{j = 0}^{k-2} (j + 1)y^j x^{k - 2 - j} $, \\
	$ Q_k(x, y) = \sum_{j = 0}^{j = k - 1} x^{j}y^{k - 1 - j} $
	$ a = 2r_2 + |\phi|_L^2 $, $ b = 2r_2 + 2 - |\phi|_L^2 $, \\
	$ G = -D^{1,0} \phi \wedge D^{0,  1} \phi^* $.
	\end{lem}

In particular,
$$ \Theta_h^{n + 1} = \begin{pmatrix}
	  -iP_{n+1, 1}(\Theta_1, \Theta_2 ) \wedge \omega_{FS} & 0 \\
	0 &  -iP_{n + 1, 2}(\Theta_1, \Theta_2) \wedge \omega_{FS}
\end{pmatrix} $$ as $ Q_{n+1} $ is an $ (n, n)$-form on $ X $ and $ D^{1, 0} \beta = (D^{1, 0} \phi) \wedge \frac{dz}{(1 + |z|^2)^2} \otimes d\bar{z} $ (since $ D^{1, 0} \big( \frac{dz}{(1 + |z|^2)^2} \otimes d\bar{z} \big) = 0 $ ). \\
The above lemma holds for any values of $ a, b $ and when the $ \Theta_i $s are any $ (1, 1) $-forms on $ X $, not necessarily equal to the curvature of any line bundle over $ X $. This will be required along the continuity path we use below. \\

For future use, we also note down
$$
 \Theta_h =
 \begin{pmatrix}
 \Theta_{1} - ia \omega_{FS} & D^{1, 0} \beta \\
 -D^{0, 1} \beta^{*} & \Theta_{2} - ib \omega_{FS}
 \end{pmatrix} $$and
$$
\Theta_h^2 =
\begin{pmatrix}
\Theta_{1}^2 - i(2a \Theta_{1} + G) \wedge \omega_{FS} & (\Theta_{1} + \Theta_{2}) \wedge  D^{1, 0} \beta \\
-(\Theta_{1} + \Theta_{2} ) \wedge D^{0, 1} \beta^{*} & \Theta_{2}^2 - i(2b \Theta_{2} + G) \wedge \omega_{FS}
\end{pmatrix}.
$$
	
	Before going on to prove the preservation of (restricted)-MA-positivity for the vortex bundle when $ M $ is a three-fold, we digress to illustrate how the MA-positivity condition can be used in potentially finding solutions to the vbMA equation. We wish to solve the equation
\begin{gather}
  (i \Theta_h)^{n+1}  = \eta \otimes \text{Id},
  \label{vbMAinthevortexcase}
  \end{gather}
where $\eta$ is a given invariant volume form in the cohomology class $\frac{(n + 1)!(2\pi)^{n+1} ch_{n+1}(V)}{2}$ (where $ch_k$ is the $k^{th}$ Chern character class). By Yau's solution of the Calabi conjecture \cite{Yau78} there exists a form $\Omega=\omega_X+\frac{1}{2n \pi}\omega_{FS} \in  c_1(L)+ \frac{1}{2n\pi}[\omega_{FS}]$ such that $\eta=c_0 \Omega^{n+1}$ where $c_0=\frac{\int \eta}{\int \Omega^{n+1}}$. \\
\indent    To solve Equation \eqref{vbMAinthevortexcase}, consider the continuity path (as in the introduction)
    $$  (i \Theta_{h_t} + t \Omega \otimes \text{Id})^{n + 1} = c_t \Omega^{n + 1} \otimes \text{Id}, $$ where
	$c_t = \int_M \text{Tr} (i \Theta_0 + t \Omega \otimes \text{Id})^{n + 1} / (2 \int_M \Omega^{n + 1}) $ is a polynomial of degree $ n + 1 $ in $ t $. That is, we consider the $\mathbb{R}$-vector bundle $V\otimes L^{t/{2\pi}} \otimes \mathcal{O}\left(\frac{t}{2n\pi} \right)$ and solve the vbMA equation for it with a normalised right-hand-side.
	\\	
	For $ t \geq 0  $ and $ h $ an invariant metric, let us define $ F(t, h) := ( i \Theta_h + t \Omega \otimes \text{Id})^{n + 1} - c_t \Omega^{n + 1} \otimes \text{Id} $. Let $ I $ denote the set of all $ t \geq 0 $ such that we have a solution $ h_t $ with $ F(t, h_t)  = 0 $ and $ i \Theta_h + t \Omega \otimes \text{Id} $ satisfying MA-positivity
	\\
	
	\begin{prop}
			I is non-empty and open
	\end{prop}
	
	We sketch a proof of this proposition. Openness follows from MA-positivity as mentioned earlier with only minor modifications. To show that $ I $ is non-empty we prove that $ I $ contains $ t $ for all $ t $ large enough. This can be done by considering the map $$ G(s, H) := s^n F(1/s, H) = (( i s \Theta_h + \Omega \otimes \text{Id})^{n + 1} - s^{n + 1} c_{1/s} \Omega^{n + 1} \otimes \text{Id})/s. $$ $ G $ is smooth for $ s \geq 0 $ and $ G(0, H) = i n\Theta_h \wedge \Omega^{n} - d \cdot \Omega^{n + 1} \otimes Id $ where $ d =\frac{\int 2\pi c_1(V)n \Omega^n}{2\int \Omega^{n+1}}$. \\
\indent The equation $G(0,H)=0$ is the Hermitian-Einstein equation. If it has a solution, then the corresponding metric can be chosen to be invariant (by averaging). Theorem 4 of \cite{Gar93} shows that it is enough to test stability against invariant coherent subsheaves with torsion-free quotients. Since stability is unchanged under tensoring with a line bundle, by considering $V\otimes L^{-r_1} \otimes \mathcal{O}(-2r_2)$ we can reduce to the case of the standard vortex bundle considered in \cite{Gar93} and apply Theorem 13 of \cite{Gar93} to conclude that indeed the bundle $V$ is stable (essentially, it is enough to test stability for the invariant subbundle $V_1$). \\
\indent Furthermore, the linearisation of the map $G(s,H)$ can be shown to be invertible at $s=0$ and hence the implicit function theorem can be used to obtain solutions for all $ s $ near zero. Since $ i s \Theta_{h_s} + \Omega \otimes \text{Id} $ is $ MA $-positive at $ s = 0 $ and $ h_s $ depends smoothly on $ s $, it follows that the solutions nearby are MA-positive as well. \qed \\
	
	As noted earlier, due to the symmetry of the metrics we are using here, we can restrict the use of the MA-positivity condition to a proper subspace of endomorphism-valued $ (1, 0) $-forms. Specifically, we consider a proper subspace ($W$ in the notation of Section \ref{sec:MApositivity}) which contains all $ D^{1, 0}g $ where $ g \in \text{End}(V_1, V_2) $ is diagonal as above. These are seen to be of the form $$ A =
	\begin{pmatrix}
	\eta & 0 \\
	\zeta & \xi
	\end{pmatrix}
	$$where $ \xi, \eta $ are $ (1, 0) $-forms on $ X $ and $ \zeta $ is a $ (1, 0) $-form on $ \mathbb{CP}^1 $ taking values in $ \text{End}(V_1, V_2) $, which we take as the restrictions defining the subspace $ W $.  With this restriction, we have the following result.
	\\
	
	\begin{thm}
		When $ \text{dim}(M) = 3 $, restricted MA-positivity is preserved along a continuity path
	\end{thm}
	
	In the remaining part of this section, we prove the above theorem. The vbMA equations when $ n = \text{dim}(X) = 2 $ can be written as
	\begin{gather}
		3a (i\Theta_1)^2 + iG \wedge (2 i\Theta_1 + i\Theta_2) = \mu \omega_X^2 \\
		3b (i\Theta_2)^2 + iG \wedge (2 i\Theta_2 + i\Theta_1) = \mu \omega_X^2
		\label{vbMA3d}
	\end{gather}
for some positive constant $ \mu $ with $ \Theta_1, \Theta_2, G, a, b, \omega_X $ as before. Recall $ a = 2r_2 + |\phi|_L^2 $ and $ b = 2r_2 + 2 - |\phi|_L^2 $\\

We now rewrite the restricted $ MA $-positivity condition in this case. Let $$ A =
\begin{pmatrix}
\eta & 0 \\
\zeta & \xi
\end{pmatrix}
$$be as before with $ \eta, \xi $ (1, 0)-forms on $ X $ and $ \zeta $ an $ \text{End}(V_1, V_2) $-valued $ (1, 0) $-form on $ \mathbb{CP}^{1} $. Defining $ \sigma_{FS} := -i \omega_{FS} $, we have as noted earlier
$$ \Theta =
\begin{pmatrix}
\Theta_1 + a\sigma_{FS} & D^{1, 0} \beta \\
-D^{0, 1} \beta^{*} & \Theta_2 + b \sigma_{FS}
\end{pmatrix} $$ and $$ \Theta^2 =
\begin{pmatrix}
\Theta_1^2 + (2a \Theta_{1} + G) \wedge \sigma_{FS} & (\Theta_{1} + \Theta_{2})\wedge D^{1, 0} \beta \\
-(\Theta_{1} + \Theta_{2})\wedge D^{0, 1} \beta^{*} & \Theta_2^{2} +  (2b \Theta_{2} + G) \wedge \sigma_{FS}
\end{pmatrix}. $$ \\

The restricted MA-positivity condition for $ \Theta $ is $$ \displaystyle \text{Tr}( i A \wedge A^* \wedge ( i \Theta)^2) + \text{Tr}( i A \wedge i\Theta \wedge  A^* \wedge  i \Theta) + \text{Tr}( i A  \wedge ( i \Theta)^2 \wedge A^* ) > 0. $$ We now compute each of these terms. \\

Let $ \gamma $ denote the $ \mathcal{O}(2)$-valued $ (1, 0) $-form $ \frac{\partial }{\partial z} \otimes dz $ on $ \mathbb{CP}^{1} $. As seen before, the adjoint $ \gamma^* $ is $ \frac{dz}{(1 + |z|^{2})^{2}} \otimes d \bar{z} $. Suppose, $ e $ is a local section of $ L $ with $ |e|_L = 1 $. Then $ D^{1, 0} \phi  = e \otimes \ell $ for some (1, 0)-form $ \ell $ on $ X $ and $ D^{0, 1} \phi^{*} = e^* \otimes \bar{\ell} $, where $ e^{*} $ is the dual of $ e $ (which is also the adjoint of $ e $ since $ e $ has unit norm). In terms of $ \gamma $ and $ \ell $, we see that $ D^{1, 0} \beta =  \ell \wedge (e \otimes \gamma^*) $ and  $ D^{0, 1} \beta^* = \bar{\ell} \wedge (\gamma \otimes e^*) $. \\

Since $ \zeta $ is section of $ \text{Hom}(\pi_1^{*}L, \pi_2^*\mathcal{O}(2)) \otimes \pi_2^{*}\mathcal{O}(-2) $, at a given point, it is equal to $ k (\gamma \otimes e^*) $ for some complex number $ k $ and hence $ \zeta^{*} = \bar{k}  (e \otimes \gamma^*) $. So we have $ \zeta \wedge D^{1, 0 }\beta = -k \ell \wedge \sigma_{FS} $ and $ \zeta^{*} \wedge D^{ 0, 1 }\beta^{*} = \bar{k} \bar{\ell} \wedge \sigma_{FS} $. With these preliminary observations, we proceed with the computation of the various traces above. \\

We have
$$
A \wedge A^{*} = \begin{pmatrix}
\eta \wedge \bar{\eta} & \eta \wedge \zeta^* \\
\zeta \wedge \bar{\eta} & \zeta \wedge \zeta^{*} + \xi \wedge \bar{\xi}
\end{pmatrix}
$$so by the expression for $ \Theta^2 $ (we do not need to compute all the entries of $  A \wedge A^{*} \wedge \Theta^2$ because we just need its trace)
\begin{multline*}
	\Tr(A \wedge A^{*} \wedge \Theta^2) = \sigma_{FS} \wedge ((2a \Theta_{1} + G)\wedge \eta \wedge \bar{\eta} - \bar{k} \eta \wedge \bar{\ell} \wedge (\Theta_{1} + \Theta_{2}) \\
	+ k \bar{\eta} \wedge \ell \wedge (\Theta_{1} + \Theta_{2}) + |k|^2 \Theta_{2}^{2} + (2b \Theta_{2} + G) \wedge \xi \wedge \bar{\xi}
	)
	$$,
\end{multline*}where we have used identities of the form $ \eta \wedge \Theta_{1}^{2} = 0 $, $ \zeta \wedge \sigma_{FS} = 0 $, etc. \\

Next,
$$
A \wedge \Theta^{2} = \begin{pmatrix}
\eta \wedge (2a \Theta_{1} + G) \wedge \sigma_{FS} & 0 \\
\zeta \wedge \Theta_{1}^{2} - \xi \wedge D^{0, 1} \beta^* \wedge (\Theta_{1} + \Theta_{2}) & -k \ell \wedge \sigma_{FS} \wedge (\Theta_{1} + \Theta_{2}) + \xi \wedge \sigma_{FS} \wedge (2b\Theta_{2} + G)
\end{pmatrix}
$$so that
\begin{multline*}
	\Tr(A  \wedge \Theta^2 \wedge A^{*}) = \sigma_{FS} \wedge ( |k|^{2} \Theta_{1}^{2} - \bar{k} \xi \wedge \bar{\ell} \wedge (\Theta_{1} + \Theta_{2}) \\ + k \bar{\xi} \wedge \ell \wedge (\Theta_{1} + \Theta_{2})  + \xi \wedge \bar{\xi} \wedge (2b \Theta_{2} + G) + \eta \wedge \bar{\eta} \wedge (2a \Theta_{1} + G)
	)
\end{multline*}

Lastly,
$$
A \wedge \Theta = \begin{pmatrix}
\eta \wedge (\Theta_{1} + a \sigma_{FS}) & \eta \wedge D^{1, 0} \beta \\
\zeta \wedge \Theta_{1} - \xi \wedge D^{0, 1} \beta^* & \zeta \wedge D^{1, 0} \beta + \xi \wedge (\Theta_{2} + b \sigma_{FS})
\end{pmatrix}
$$and
$$
A^* \wedge \Theta = \begin{pmatrix}
\bar{\eta} \wedge (\Theta_{1} + a \sigma_{FS}) - \bar{k} \bar{\ell} \wedge \sigma_{FS} & \bar{\eta} \wedge D^{1, 0} \beta + \zeta^* \wedge \Theta_{2} \\
- \bar{\xi} \wedge D^{0, 1} \beta^* & \bar{\xi} \wedge (\Theta_{2} + b \sigma_{FS})
\end{pmatrix}
$$so
\begin{multline*}
	\Tr(A  \wedge \Theta \wedge A^{*} \wedge \Theta) = \sigma_{FS} \wedge ( \eta \wedge \bar{\eta} \wedge  2a \Theta_{1} - \bar{k} \eta \wedge \bar{\ell} \wedge \Theta_1 \\
	+ G \wedge \eta \wedge \bar{\xi} + k \bar{\eta} \wedge \ell \wedge \Theta_{1} + |k|^{2} \Theta_{1} \Theta_{2} + G \wedge \xi \wedge \bar{\eta} - \bar{k} \xi \wedge \bar{\ell}  \wedge \Theta_{2} \\
	+ k \bar{\xi} \wedge \ell \wedge \Theta_{2} + \xi \wedge \bar{\xi}  \wedge 2b \Theta_{2}
	)
\end{multline*}

Adding these, we see that the restricted MA-positivity condition can be written as
	\begin{multline}
		|k|^2((i\Theta_1)^2 + (i\Theta_2)^2 + i\Theta_1 \wedge  i\Theta_2) \\
		+ (2 i\Theta_1 + i\Theta_2) \wedge (ik \bar{\eta} \wedge \ell + i\bar{k} \bar{\ell} \wedge \eta ) + (2 i\Theta_2 + i\Theta_1) \wedge (ik \bar{\xi} \wedge \ell + i\bar{k} \bar{\ell} \wedge \xi ) \\
		+ 2(3a i\Theta_1 + iG) \wedge i\eta \wedge \bar{\eta} + 2(3b i\Theta_2 + iG) \wedge i\xi \wedge \bar{\xi}  + iG \wedge (i\eta \wedge \bar{\xi} + i\xi \wedge \bar{\eta})\\ > 0
		\label{rMApos}
	\end{multline}for $ k \in \mathbb{C} $, $ \eta, \xi \in T^{1, 0}(X) $ not all zero. Note that by the definition of $ \ell $ above, $ G = - \ell \wedge \bar{\ell} $.  Without loss of generality, $ k $ can be taken to be real as any arguments of $ k $ can be absorbed within $ \eta, \xi $ due to the homogeneity of the left-hand-side in the three variables $ k, \xi, \eta $. \\

	In the following part of this section, for clarity of notation, we omit factors of $ i $ in $ i\Theta_1, i\Theta_{2}, iG $ and the terms of type $ i \eta \wedge \xi $, etc. coming from $ iA \wedge A^* $. 	
	\begin{lem}
		 $ a, b > 0 $ on $ X $.
\end{lem}
\begin{proof}
		That $ a > 0 $ follows immediately from the expression $ a = 2r_2 + |\phi|_L^2 $. Pick a point $ p $ on the manifold where $ |\phi|_L(p) $ is maximum and hence $ b = 2r_2 + 2 - |\phi|_L^2 $ is minimum. At this point $ D^{1, 0} \phi(p) = 0 $ and hence $ G = -D^{1, 0} \phi \wedge D^{0, 1} \phi^\dagger = 0 $. By the vbMA equation, $ 3b \Theta_2^2 > 0 $ at $ p $, so $ b \neq 0 $ and $ \Theta_2^{2} \neq 0 $ at $ p $. By MA-semi-positivity, $ 3b \Theta_2 \geq 0 $ and hence $ \Theta_2^2 \geq 0 $, so we have $ b > 0 $ at $ p $ and hence on $ X $. \\
	\end{proof}

	\begin{lem}
		$ 3a \Theta_1 + G > 0 $, $ 3b\Theta_2 + G > 0 $.
	\end{lem}
	\begin{proof}
		 Consider $ 3a \Theta_1 + G $. That this real $(1, 1)$-form is positive semi-definite follows immediately from the MA-semi-positivity condition (taking $ k = 0 $ and $ \xi = 0 $) so it is sufficient to show that $ (3a\Theta_1 + G)^2 = 3a(3a\Theta_1^2 + 2G\Theta_1) > 0 $. This inequality follows directly from the vbMA equation \eqref{vbMA3d}, the fact that $ \Theta_i \geq 0 $, $ a, b > 0$, $G \leq 0 $ and $ G $ having rank 1 (because we then have $ G \Theta_i  \leq 0 $) so we are done. Similarly, it can be shown that $ 3b\Theta_2 + G > 0 $. \\
	\end{proof}

	To simplify the MA-positivity condition, we write $ \widetilde{\Theta_1} = 3a \Theta_1 + G$ and $ \widetilde{\Theta_2} = 3b\Theta_2 + G $. In terms of these newly-defined positive $ (1, 1) $-forms, the vbMA equation becomes
	\begin{gather}
		b\widetilde{\Theta_1}^2 + a G \widetilde{\Theta_2} = 3ab \mu \omega_X^2 \\
		a\widetilde{\Theta_2}^2 + b G \widetilde{\Theta_1} = 3ab \mu \omega_X^2
	\end{gather}
	and the restricted MA-positivity condition becomes
	\begin{multline}
		k^2(b^2 \widetilde{\Theta_1}^2 + a^2 \widetilde{\Theta_2}^2 + ab\widetilde{\Theta_1} \widetilde{\Theta_2} -G(b(2b + a) \widetilde{\Theta_1} + a(2a + b) \widetilde{\Theta_2})) \\
		+ k(2b \widetilde{\Theta_1} + a\widetilde{\Theta_2})( \bar{\eta} \wedge \ell + \bar{\ell} \wedge \eta ) + k(2a\widetilde{\Theta_2} + b\widetilde{\Theta_1})( \bar{\xi} \wedge \ell + \bar{\ell} \wedge \xi ) \\
		+ 2\widetilde{\Theta_1}\eta \wedge \bar{\eta} + 2\widetilde{\Theta_2} \xi \wedge \bar{\xi}  + G(\eta \wedge \bar{\xi} + \xi \wedge \bar{\eta})\\ > 0, \label{3dMApos}
	\end{multline}where we have taken $ k $ to be real and as before not all of $k, \xi, \eta  $ are zero. \\
	
	As this is a point-wise inequality, we choose appropriate coordinates at an arbitrary point $ p $ such that $ \widetilde{\Theta_1} = dz^1 \wedge d\bar{z}^1 + dz^2 \wedge d\bar{z}^2  $, $ \widetilde{\Theta_2} = \lambda_1 dz^1 \wedge d\bar{z}^1 + \lambda_2 dz^2 \wedge d\bar{z}^2 $ with $ \lambda_i > 0 $. Suppose $ \ell = \ell_1 dz^1 + \ell_2 dz^2 $, $ \xi = \xi_1 dz^1 + \xi dz^2 $ and $ \eta = \eta_1 dz^1 + \eta_2 dz^2 $ in these coordinates. The vbMA equation at $ p $ can then be written as
	\begin{gather} 
		2b - a(\lambda_1 |\ell_2|^{2} + \lambda_2 |\ell|^{2}) = c \\
		2a\lambda_1 \lambda_2 - b(|\ell_1|^{2} + |\ell_2|^{2}) = c,
		\label{cooridinatevbMA}
	\end{gather}
	with $ c $ being positive. \\

	By replacing $k$ with $-k$, the MA-positivity condition at $ p $ becomes
	\begin{equation}
		k^2 \Delta + 2k \Re{\langle v, w \rangle} + \langle Mv, v \rangle > 0,
	\end{equation}where $ \Delta $ is the (positive) coefficient of $ k^2 $ in $ \eqref{3dMApos} $ divided by the volume form (omitting factors of $ i $ as usual) $ dz^1 \wedge d\bar{z}^1 \wedge dz^2 \wedge d\bar{z}^2 $,
	$$ M = \begin{bmatrix}
	 2 & 0 & -|\ell_2|^2 & \ell_1 \bar{\ell_2} \\
	 0 & 2 & \ell_2 \bar{\ell_1} & -|\ell_1|^2 \\
	 -|\ell_2|^2 & \ell_1 \bar{\ell_2} & 2 \lambda_2 & 0 \\
	 \ell_2 \bar{\ell_1} & -|\ell_1|^2 & 0 & 2 \lambda_1 \\
	\end{bmatrix}, $$
	$$ v = \begin{bmatrix}
	\eta_1 \\ \eta_2 \\ \xi_1 \\ \xi_2
	\end{bmatrix}, $$
	\begin{equation}
	 \label{greek} w = \begin{bmatrix}
	\alpha \\ \beta \\ \gamma \\ \delta
	\end{bmatrix} := \begin{bmatrix}
	\ell_1(2b + a \lambda_2) \\ \ell_2(2b + a \lambda_1) \\ \ell_1(b + 2a \lambda_2) \\ \ell_2(b + 2a \lambda_1),
	\end{bmatrix}
	\end{equation}and $ \langle , \rangle $ is the usual inner product on $ \mathbb{C}^{4} $. \\
	
	Since $ \Delta > 0 $, $ \Theta $ is MA-positive in the restricted sense iff we have $$ \Delta \langle Mv, v \rangle > |\Re \langle v, w \rangle |^2 $$ for all $ v \neq 0 $. As $ v $ is complex, the complex phase of $ \langle v, w \rangle  $ can be absorbed into $ v $ so that the inequality becomes $\Delta \langle Mv, v \rangle > | \langle v, w \rangle |^2 $ which can be rewritten as the condition that $$ X := \Delta M - w w^* \label{Xpositivity} $$ is positive definite, where $ w^* $ is the conjugate transpose of $ w $. By MA-semi-positivity, we know that $ X $ is already positive semi-definite, so it suffices to show that $ \det(X) > 0 $.
	\\
	
	$ X $ can be written as
	$$ X = \begin{bmatrix}
	A & B \\ C & D
	\end{bmatrix} = \begin{bmatrix}
	2\Delta \text{Id}_2 - w_1 w_1^* & - \Delta ll^* - w_1 w_2^* \\
	-\Delta ll^* - w_2 w_1^* & \Lambda \Delta - w_2 w_2^*
	\end{bmatrix}
	$$where $ w_1 = \begin{bmatrix}
	\alpha \\ \beta
	\end{bmatrix} $, $ w_2 = \begin{bmatrix}
	\gamma \\ \delta
	\end{bmatrix} $ and $ l = \begin{bmatrix}
	\bar{\ell_2} \\ -\bar{\ell_1}
	\end{bmatrix} $ and $ \Lambda = \begin{bmatrix}
	2 \lambda_2 & 0 \\ 0 & 2 \lambda_1
	\end{bmatrix} $.\\
	
	\begin{lem}
	 We have $ \det(A) = 2 \Delta (2 \Delta - |w_1|^{2}) > 0 $, where $ |w_1| $ denotes the Euclidean norm of the column vector $ w_1 $.
	\end{lem}
	\begin{proof}
	
	\begin{multline*}
	2 \Delta - |w_1|^2 = 4b^2 + 4a^2(\lambda_1 \lambda_2 + \lambda_2 |\ell_1|^{2} + \lambda_1 |\ell_2|^{2}) + \\
	2ab(\lambda_1 + \lambda_2 + |\ell|^{2} ) -a^2(\lambda_1^{2} |\ell_2|^{2} + \lambda_2^{2} |\ell_1|^{2})	- 2ab(\lambda_1 |\ell_2|^{2} + \lambda_2 |\ell_1|^{2})
	\label{detAfactor}
	\end{multline*}
	By $ \eqref{cooridinatevbMA} $, $  4b^2 - 2ab(\lambda_1 |\ell_2|^{2} + \lambda_2 |\ell_1|^{2}) = 2b(2b - a(\lambda_1 |\ell_2|^{2} + \lambda_2 |\ell_1|^{2})) > 0  $ and $ 2ab(\lambda_1 + \lambda_2) > a^2(\lambda_1 + \lambda_2)(\lambda_1 |\ell_2|^{2} + \lambda_2 |\ell_1|^{2}) > a^2(\lambda_1^{2} |\ell_2|^{2} + \lambda_1^{2} |\ell_1|^{2}) $ so $ \det(A) > 0 $ as required. \\
\end{proof}
	
	Consequently, $ A $ is invertible and $ \det(X) = \det(A) \det(D - CA^{-1}B) $. By a direct computation,
	\begin{equation*}
		D - CA^{-1}B = \frac{\Delta}{2k} (2k\Lambda - 4w_2w_2^{*} - (k|\ell|^{2} + |\langle l, w_1 \rangle |^{2})ll^{*} -2(\langle w_1, l \rangle lw_2^* + \langle l, w_1 \rangle w_2l^*)  )
	\end{equation*}
	where $ k = 2\Delta - |w_1|^{2} $, $ |\ell|^{2} = |\ell_1|^{2} + |\ell_2|^{2} $ and $ \langle l, w_1 \rangle  = w_1^* l $. \\
	
	The computation of the determinant of a matrix of this form can be simplified by considering it as a real $ (1, 1) $-form. In fact, if $ p, q $ are two $ 1, 0 $-forms on $ \mathbb{C}^{2} $ and $ \eta $ is a $ (1, 1) $-form and $ \omega = \eta + ap \wedge \bar{p} + b q \wedge \bar{q} + c p \wedge \bar{q} + \bar{c} q \wedge \bar{p} $ with $ a, b \in \mathbb{R} $ and $ c \in \mathbb{C} $, then $ \omega^{2} = \eta^{2} + 2 \eta (\omega - \eta) + 2(ab - |c|^{2}) p \wedge \bar{p} \wedge q \wedge \bar{q} $. Dividing throughout by a standard volume form gives an expression for the determinant of the matrix corresponding to $\omega  $. Using this, we find
	
	\begin{multline}
		\frac{\det(X)}{2 \Delta^3} =  \color{blue}
		4\lambda_1 \lambda_2 (2 \Delta - |w_1|^2 ) \color{purple}  +
		|\bar{\ell_2}\delta + \bar{\ell_1} \gamma |^2|\ell|^2 \\
		- \lambda_1( \color{purple} 4 |\gamma|^{2} + \color{blue} (2 \Delta - |w_1|^2)|\ell|^2 |\ell_2|^2 + \color{red} |\ell_2|^2|\ell_2 \alpha - \ell_1 \beta|^{2} +  4 \Re((\ell_2 \alpha - \ell_1 \beta)\bar{\ell_2} \bar{\gamma})) \\
		- \lambda_2(\color{purple} 4 |\delta|^{2} + \color{blue} (2 \Delta - |w_1|^2)|\ell|^2 |\ell_1|^2 + \color{red} |\ell_1|^2|\ell_2 \alpha - \ell_1 \beta|^{2} - 4 \Re((\ell_2 \alpha - \ell_1 \beta)\bar{\ell_1} \bar{\delta}))
		\label{detX}
	\end{multline}
	As $ \Delta > 0 $, it is enough to show that the right-hand side of the above equation is positive. \\
	
	We define the following quantities
	\begin{equation}
	c_1 = 2b - a\lambda_1 |\ell_2|^{2} - a \lambda_2 |\ell_1|^{2}
	\end{equation}
	\begin{equation}
	c_2 = 2a \lambda_1 \lambda_2 - b |\ell|^{2}
	\end{equation}
	\begin{equation}
	c_3 = 4 \lambda_1 \lambda_2 - \lambda_1 |\ell_2|^{2}|\ell|^{2} - \lambda_2 |\ell_1|^{2}|\ell|^{2}
	\end{equation}
	and note that
	\begin{equation}
	c_3 = 2\frac{c_2}{a} + \frac{c_1 |\ell|^{2}}{a}
	\end{equation}
	By the vbMA equations \eqref{cooridinatevbMA}, we know that each $ c_i > 0 $, so our strategy will be to write the right-hand-side of $ \eqref{detX} $ in terms of the $ c_i $. \\
	
	The $ \textcolor{blue}{blue} $ terms when grouped together clearly contain a factor of $ c_3 $:
	\begin{equation}
	\label{blue1}
	(4 \lambda_1 \lambda_2 - \lambda_1 |\ell_1|^{2}|\ell|^{2} - \lambda_2 |\ell_2|^{2}|\ell|^{2})(2 \Delta - |w_1|^{2}).
	\end{equation}
	
	The sum of the $ \textcolor{purple}{purple} $ terms is
	\begin{multline}
	\label{purple1}
	-|\delta|^2(4\lambda_2 - |\ell_2|^{2} |\ell|^{2}) - |\gamma|^{2}(4\lambda_1 - |\ell_1|^{2} |\ell|^{2}) \\
	+ 2 |\ell_2|^{2}|\ell_1|^{2}(b^2 + 4a^{2}\lambda_1 \lambda_2 + 2ab(\lambda_1 + \lambda_2))|\ell|^{2},
	\end{multline}where we have used the expressions in $ \eqref{greek} $ for $ \gamma, \delta $ while expanding $ |\bar{\ell_1} \gamma + \bar{\ell_2} \delta|^{2} $. We see that the coefficients of $ |\delta|^{2} $ and $ |\gamma|^{2} $ here can be made proportional to $ c_3 $ by adding and subtracting certain terms. Doing this, we find that \eqref{purple1} is equal to the sum of two parts:
	
	\begin{equation}
	\label{purplesplit1}
	(4 \lambda_1 \lambda_2 - \lambda_1 |\ell_1|^{2}|\ell|^{2} - \lambda_2 |\ell_2|^{2}|\ell|^{2}) \times (- \lambda_1 |\gamma|^{2} - \lambda_2 |\delta|^{2}) \times \frac{1}{\lambda_1 \lambda_2}
	\end{equation}
	and
	\begin{equation}
	\label{purplesplit2}
	2 |\ell_2|^{2}|\ell_1|^{2}(b^2 + 4a^{2}\lambda_1 \lambda_2 + 2ab(\lambda_1 + \lambda_2))|\ell|^{2} - |\delta|^{2} |\ell|^{2} |\ell_1|^{2} \frac{\lambda_2}{\lambda_1} - |\gamma|^{2} |\ell_2|^{2} |\ell|^{2} \frac{\lambda_1}{\lambda_2}
	\end{equation}
	
	From the expressions \eqref{greek} for $ \delta, \gamma $, \eqref{purplesplit2} can be simplified to give
	\begin{equation}
	\label{purple2}
	\frac{-b^{2}(\lambda_1 - \lambda_2)^{2}}{\lambda_1 \lambda_2} |\ell_1|^{2} |\ell_1|^{2} |\ell|^{2}.
	\end{equation}Now, \eqref{purplesplit1} can be added to \eqref{blue1} to get
	\begin{equation}
	\label{bluepurple1sum}
	(4 \lambda_1 \lambda_2 - \lambda_1 |\ell_1|^{2}|\ell|^{2} - \lambda_2 |\ell_2|^{2}|\ell|^{2}) \times (\lambda_1 \lambda_2 (2 \Delta - |w_1|^{2}) - \lambda_1 |\gamma|^{2} - \lambda_2 |\delta|^{2}) \times \frac{1}{\lambda_1 \lambda_2}.
	\end{equation}
	
	Writing $ c_3 $ in terms of $ c_1 $ and $ c_2 $, \eqref{bluepurple1sum} is equal to
	\begin{equation}
	\label{bluepurple1sum2}
	\left( 2\frac{c_2}{a} + \frac{c_1 |\ell|^{2}}{a} \right) \frac{f}{\lambda_1 \lambda_2}
	\end{equation}
	where
	\begin{multline}
	f = 4 \lambda_1 \lambda_2 b^{2} + 4\lambda_1^{2}\lambda_2^{2}a^{2} + 2ab\lambda_1 \lambda_2(\lambda_1 + \lambda_2 - |\ell|^{2} - \lambda_1|\ell_2|^{2} - \lambda_2 |\ell_1|^{2}) \\
	-a^{2}\lambda_1 \lambda_2 (\lambda_1^{2} |\ell_2|^{2} + \lambda_2^2 |\ell_1|^{2}) -b^{2}(\lambda_1 |\ell_1|^{2} + \lambda_2 |\ell_2|^{2}).
	\end{multline}Thus, the sum of the blue and purple terms is equal to the sum of \eqref{bluepurple1sum2} and \eqref{purple2}. \\
	
	To simplify the remaining $ \textcolor{red}{red} $ terms in \eqref{detX}, we use the expressions for $ \alpha, \beta, \gamma, \delta $ to find
	\begin{gather}
	\label{computations1}
		\ell_2 \alpha - \ell_1 \beta = \ell_1 \ell_2 a(\lambda_2 - \lambda_1) \\
		(\ell_2 \alpha - \ell_1 \beta) \bar{\ell_2} \bar{\gamma} = |\ell_1|^{2}|\ell_2|^{2}a(\lambda_2 - \lambda_1)(b + 2a \lambda_2) \\
		(\ell_2 \alpha - \ell_1 \beta) \bar{\ell_1} \bar{\delta} = |\ell_1|^{2}|\ell_2|^{2}a(\lambda_2 - \lambda_1)(b + 2a \lambda_1)
	\end{gather}
	so that the sum of the red terms in $ \eqref{detX} $ is
	\begin{multline}
		-4(\lambda_1 (\ell_2 \alpha - \ell_1 \beta) \bar{\ell_2} \bar{\gamma} -
		  \lambda_2 (\ell_2 \alpha - \ell_1 \beta) \bar{\ell_1} \bar{\delta}) - (\lambda_2|\ell_1|^2 + \lambda_1 |\ell_2|^{2})|\ell_2 \alpha - \ell_1 \beta|^{2}  \\ =
		  |\ell_1|^{2}|\ell_2|^{2}a(\lambda_2 - \lambda_1)^{2}(4b - a (\lambda_1 |\ell_2|^{2} + \lambda_2 |\ell_1|^{2}))		
		  \label{partialresult1}.
	\end{multline}

	This can be added to \eqref{purple2}, resulting in
	\begin{multline}
		|\ell_1|^{2}|\ell_2|^{2}a(\lambda_1 - \lambda_2)^{2}(2b - a\lambda_1 |\ell_2|^{2} - a \lambda_2 |\ell_1|^{2}) \\
		+ |\ell_1|^{2}|\ell_2|^{2}b(\lambda_1 - \lambda_2)^{2}(2a - b \frac{|\ell|^{2}}{\lambda_1 \lambda_2})
		\label{partialresult}.
	\end{multline}
	
	Finally, \ref{bluepurple1sum2} is added to  $ \ref{partialresult} $ to get:
	\begin{multline}
		\frac{\det(X)}{2 \Delta^{3}} = \frac{c_2}{\lambda_1 \lambda_2}\left(2\frac{f}{a} + |\ell_1|^{2}|\ell_2|^{2}b(\lambda_1 - \lambda_2)^{2} \right) + c_1\left(\frac{f |\ell|^{2}}{a \lambda_1\lambda_2} + |\ell_1|^{2}|\ell_2|^{2}a(\lambda_1 - \lambda_2)^{2} \right) \\
		=: \frac{c_2 g_2}{a \lambda_1 \lambda_2} + \frac{c_1 g_1}{a \lambda_1\lambda_2},
		\label{tobeprovedpositive}
	\end{multline} with $ g_i $ being the coefficients of $ c_i $ in the left-hand-side, up to certain positive factors. \\
	
	\begin{lem}
		The coefficients $ g_i $ above are positive.
	\end{lem}
	Proving this lemma will complete the proof of the Theorem since
	we know that $ c_i > 0 $, from which it follows that $ \det(X) > 0 $.
	\begin{proof}

	Note that (with $ a, b $ and the $ \lambda_i $s fixed) the $ g_i $s are superharmonic functions of $ |\ell_1|^{2} $ and $ |\ell_2|^{2} $ and
	$ f $ is affine in those same variables. Due to the vbMA equation $ \eqref{cooridinatevbMA} $, $ |\ell_1|^{2}, |\ell_2|^{2} $ are constrained to lie in the bounded region $ P $
	\begin{gather*}
		|\ell_i|^2 \geq 0 \\
		\lambda_1 |\ell_2|^{2} + \lambda_2 |\ell_1|^{2} < \frac{2b}{a} \\
		|\ell_1|^{2} + |\ell_2|^{2} < \frac{2a \lambda_1 \lambda_2}{b}.
	\end{gather*}
	
	We consider the two cases $ \lambda_1 = \lambda_2 $ and $ \lambda_1 \neq \lambda_2 $ separately.
	
	$\textbf{Case 1}: \lambda_1 = \lambda_2 $\\
	In this case, it follows from the expressions for the $ g_i $ that it is enough to show $ f > 0  $ in $ P $. When $ \lambda_1 = \lambda_2 =: \lambda $,
	\begin{multline*}
		f = 4b^2 \lambda^{2} + 4ab\lambda^{3} + 4a^2\lambda^{4} - |\ell|^{2}(2ab\lambda^{2}(1 + \lambda) + a^2 \lambda^{4} + b^2 \lambda)
	\end{multline*} and the inequalities defining $ P $ are
	\begin{equation*}
		0 \leq |\ell|^{2} < \min\left( \frac{2b}{a\lambda}, \frac{2a\lambda^{2}}{b} \right).
	\end{equation*}
	 $ f(|\ell|^{2}) $ is linear and strictly-decreasing in $ |\ell|^{2} $, so if we consider the two sub-cases $ \min\left( \frac{2b}{a\lambda}, \frac{2a\lambda^{2}}{b} \right) = \frac{2b}{a\lambda} $ and $ \min\left( \frac{2b}{a\lambda}, \frac{2a\lambda^{2}}{b} \right) = \frac{2a\lambda^{2}}{b} $, then it is enough to show that $ f\left( \min\left( \frac{2b}{a\lambda}, \frac{2a\lambda^{2}}{b} \right) \right) \geq 0 $ in both cases separately. We first consider the case $ \min\left( \frac{2b}{a\lambda}, \frac{2a\lambda^{2}}{b} \right) = \frac{2b}{a\lambda} $.
	
	 $$ f\left( \frac{2b}{a \lambda} \right) = 2ab \lambda^{3} + 4a^2 \lambda^{4} - 4b^{2} \lambda - \frac{2b^{3}}{a}  = \frac{2}{a}(b + 2a\lambda)(a^2 \lambda^{3} - b^2) .$$ As we are considering the case $ \frac{2b}{a\lambda} \leq \frac{2a\lambda^{2}}{b}  $, we see that $ a^2 \lambda^{3} - b^2 \geq 0 $ so, $ f(\frac{2b}{a \lambda}) \geq 0 $ and $ f(|\ell|^{2}) > 0$ for $ |\ell|^{2} < \frac{2b}{a \lambda} $ so $ f > 0 $ on $ P $ in this case.

    Likewise, in the other case, $f\left(\frac{2a\lambda^2}{b} \right)=\frac{2\lambda^2}{b} (2b+a\lambda)(b^2-a^2\lambda^3)\geq 0$ because $\frac{2a\lambda^2}{b}\leq \frac{2b}{a\lambda}$ in this case, and hence $f>0$ on $P$.\\
	
	 $\textbf{Case 2}: \lambda_1 \neq \lambda_2$\\
	 We will show that $ \displaystyle g_2(|\ell_1|^{2}, |\ell_2|^{2}) = 2f(|\ell_1|^{2}, |\ell_2|^{2}) + ab(\lambda_1 - \lambda_2)^{2} |\ell_1|^{2} |\ell_2|^{2} > 0 $. The proof for $ g_1 $ is entirely similar. As noted before, $ g_2 $ is superharmonic, so by the strong maximum principle, it suffices to show that $ g_2{|}_{\partial P} \geq 0 $ and $ g_2 > 0 $ on parts of the boundary which also lie in $ P $. $ \overline{P} $ is a quadrilateral in the $ |\ell_1|^{2} $, $ |\ell_2|^{2} $-plane whose boundary is made up by the lines $ |\ell_1|^2 = 0 $, $ |\ell_2|^{2} = 0 $ and the two lines of negative slope $ \displaystyle \lambda_1 |\ell_2|^{2} + \lambda_2 |\ell_1|^{2} = \frac{2b}{a} $, $ \displaystyle
	 |\ell_1|^{2} + |\ell_2|^{2} = \frac{2a \lambda_1 \lambda_2}{b} $. \\
	
	 On the $ |\ell_2|^{2} = 0 $ part of the boundary,
	 \begin{multline*}
	 f = 4 \lambda_1 \lambda_2 b^{2} + 4\lambda_1^{2}\lambda_2^{2}a^{2} + 2ab\lambda_1 \lambda_2(\lambda_1 + \lambda_2 - |\ell_1|^{2} - \lambda_2 |\ell_1|^{2}) - a^{2}\lambda_1 \lambda_2^{3} |\ell_1|^{2} -b^{2} \lambda_1 |\ell_1|^{2}
	 \end{multline*}As in the $ \lambda_1 = \lambda_2 $ case, $ f $ is a strictly decreasing linear function in $ |\ell_1|^{2} $. Also, as we are only considering points on $ \partial P $, we must have $ \displaystyle |\ell_1|^{2} \leq \min \left( \frac{2b}{a\lambda_2}, \frac{2a\lambda_1\lambda_2}{b} \right) =: z_0 $, so it is enough to show that $ f(z_0) \geq 0 $ which can be done exactly as in the $ \lambda_1 = \lambda_2 $ case (for instance, when $z_0=\frac{2b}{a\lambda_2}$, $f(z_0)= \frac{2}{a\lambda_2}(2a\lambda_1\lambda_2+b\lambda_2)(\lambda_1 \lambda_2^2 a^2-b^2) \geq 0$), so $ f(z_0) \geq 0 $ and $ f(|\ell_1|^{2}, 0) > 0 $ for all points on the boundary which also lie in $ P $. Similarly, $ f $ restricted to the $ |\ell_1|^{2} = 0 $ part of the boundary is also non-negative and hence $ g_2 $ restricted to these parts of the boundary is also non-negative with $ g_2 > 0$ on parts of the boundary which also lie in $ P $. \\
	
	 When restricted to the $ \displaystyle \lambda_1 |\ell_2|^{2} + \lambda_2 |\ell_1|^{2} = \frac{2b}{a} $ part of $ \partial P $, $ g_2 $ is a concave quadratic function of an affine parameter along this line so the restriction of $ g_2 $ to this part of the boundary attains its minimum on the corners of the boundary. At one of the corners, either $ |\ell_1|^2 = 0 $ or $ |\ell_2|^2 = 0 $ and $ g_2 \geq 0 $ at those points as seen before. The other corner is given by the intersection of $ \displaystyle \lambda_1 |\ell_2|^{2} + \lambda_2 |\ell_1|^{2} = \frac{2b}{a} $ and $ \displaystyle
	 |\ell_1|^{2} + |\ell_2|^{2} = \frac{2a \lambda_1 \lambda_2}{b} $. Since $ \lambda_1 \neq \lambda_2 $, these lines intersect at a unique point
	 \begin{gather*}
	 	|\ell_1|^{2} = \frac{2(b^2 - a^2 \lambda_1^{2} \lambda_2)}{ab(\lambda_2 - \lambda_1)} \\
	 	|\ell_2|^{2} = \frac{2(b^2 - a^2 \lambda_1 \lambda_2^{2})}{ab(\lambda_1 - \lambda_2)}.
	 \end{gather*}At this intersection, $$ ab(\lambda_1 - \lambda_2)^{2} |\ell_1|^{2} |\ell_2|^{2} = \frac{4(b^2 - a^2\lambda_1^{2} \lambda_2)(a^2 \lambda_1 \lambda_2^{2} - b^2)}{ab} $$and $$ f(|\ell_1|^{2}, |\ell_2|^{2})  = \frac{2(a^2\lambda_1^{2} \lambda_2 - b^{2})(a^2 \lambda_1 \lambda_2^{2} - b^2)}{ab} $$ and hence (at this point) $ g_2 = 0 $. So we have $ g_2 \geq 0 $ on the part of the line $ \displaystyle \lambda_1 |\ell_2|^{2} + \lambda_2 |\ell_1|^{2} = \frac{2b}{a}$ which belongs to $ \partial P $. (Note that the line itself does not belong to $ P $. Therefore, we do not require $ g_2 $ to be strictly positive on any part of this line.) Similarly, it can be shown that $ g_2 \geq 0 $ along the $ \displaystyle
	 |\ell_1|^{2} + |\ell_2|^{2} = \frac{2a \lambda_1 \lambda_2}{b} $ part of the boundary as well and hence $ g_2 > 0 $ in the region $ P $ by the strong maximum principle as noted earlier. In the same manner, $ g_1 > 0 $ on $ P $. This completes the proof of case 2 and the lemma.
	\end{proof}

\end{document}